\documentclass{article}[12pt]
\usepackage{amsmath}
\usepackage{amssymb}
\usepackage{extarrows}

\newtheorem{theorem}{Theorem}

\newtheorem{proposition}[theorem]{Proposition}
\newtheorem{remark}[theorem]{Remark}

\newenvironment{proof}[1][Proof]{\noindent\textbf{#1.} }{\ \rule{0.5em}{0.5em}}
\usepackage{amssymb}
\usepackage{graphicx}
\usepackage[dvipsnames,usenames]{color}

\input{epsf.tex}
\usepackage[affil-it]{authblk}
\author{Ioannis Dimitriou \footnote{idimit@uoi.gr}\footnote{Corresponding author.}}
\affil{\small Department of Mathematics, 
	University of Ioannina, 
	45110, Ioannina, Greece.}
\begin{document}
\title{On dual risk models with proportional gains and dependencies}

\maketitle
\begin{abstract}
In this work, we consider extensions of the dual risk model with proportional gains by introducing a dependence structure between gain sizes and gain interrarrival times. Among others, we further consider the case where the proportional parameter is randomly chosen, the case where it is a uniformly random variable, as well as the case where we may have upwards as well as downwards jumps. Moreover, we consider the case with causal dependence structure, as well as the case where the dependence is based on the generalized Farlie-Gumbel-Morgenstern copula. The ruin probability and the distribution of the time to ruin are investigated.
   \end{abstract}
    \vspace{2mm}
	
\section{Introduction}
In this work, we investigate the dual risk model with constant expense rate normalized to unity, by considering several non-trivial generalizations of the model considered in \cite{boxruin}. Among others, our primary aim is to lift several independence assumptions among gain interarrival times and gain sizes, but still, to be able to obtain explicit results regarding some major metrics of interest such as ruin probability and the time to ruin.

For a detailed study on the fundamentals of ruin probabilities in the conventional, permanently inspected, Cramer-Lundberg context see \cite{asmalb,manbox}. It is well-known that there is a duality property among the ruin theory and the queueing theory. In particular, the Cramer-Lundberg model is dual to the M/G/1 queueing model with the same arrival rate and with a service time distribution that equals the claim size distribution in the Cramer-Lundberg model; see \cite{asmalb,manbox}. Quite recently, in the seminal book \cite{manbox}, the authors presented the main results and the most important probabilistic methods related to the Cramer-Lundberg model, and exploited connections with the related model in queueing theory. 

Its dual process has also attracted interest in the insurance risk literature. As pointed out in \cite{ava1}, the dual risk model describes the surplus or equity of a company with fixed expense rate and occasional income inflows of random size, called innovations or gains. These gains arise due to some
contingent events (e.g. discoveries, sales). Examples where the dual risk model applies are pharmaceutical, petroleum, or R\&D
companies. Other examples are commission-based businesses, such as real estate agents or brokerage firms that sell mutual funds or insurance products with a front-end load. As stated in \cite{boxruin}, an
illustrative realistic example of the dual risk model with proportional gains refers to start-ups or e-companies, where their gains strongly depend on the amount of investments, which is often proportional to the value of the company; e.g., the CD Projekt, one of the biggest Polish companies
producing computer games. For more information on research progress on the dual risk model and its applications, see e.g., \cite{ava1,ava2,ava3,boxmafro,fahim,afon,bayr,bayr1,hu,albtax,rodri,palmo}.

The major contribution of this work relies on the investigation of the dual risk model with proportional gain mechanism where in addition, several independence assumptions are lifted (e.g., the distribution of the gain size depends on the gain interarrival time, and this type of dependence it also affects the proportional parameter, as well as the case where the dependence structure is based on a copula), and still, we are able to derive explicit expressions for the ruin probability and the time to ruin. Our interest focuses on the derivation of the ruin probability, given that the initial surplus equals $x$, i.e., $R(x):=P(\tau_{x}<\infty|U(0)=x)$, where $\tau_{x}=\inf\{t\geq 0:U(t)=0\}$, as well as the distribution of $\tau_{x}$, i.e., the time to ruin. To investigate $R(x)$, we use a one-step analysis where the process under study is viewed at successive gain times. Given the initial capital, we obtain the Laplace transform of the ruin probability for the risk process, and the double Laplace transform of the ruin time. The approach we follow bears similarities to the method developed in \cite{box1,box2,box3,dimi,hoo} to study reflected autoregressive processes. To our best knowledge, this work provides for the first time exact expressions for ruin measures in a general dual risk model with proportional gains under dependent setting, and thus, should be viewed as a starting point for obtaining analytical results in more general dependent scenarios.

The paper is organised as follows. In Section \ref{caus} we consider the dual risk model with causal dependence structure, that relates the gain size, the gain interarrival and the proportional parameter, and obtain the Laplace transform of the ruin probability as well as the Laplace transform of the ruin time Laplace-Stieltjes transform (LST). We present result where the gain size follows exponential, Erlang or even mixed Erlang distribution. In Section \ref{copu}, we focus on the dual risk model with proportional gains and where the gain size and the gain interarrival times are dependent based on the (generalized) Farlie-Gumbel-Morgenstern (FGM) copula. We also consider the case where, in addition, there is a linear dependence among
gain interarrival times and surplus level. Section \ref{random} is devoted to the analysis of the dual risk model with randomly proportional gains with upward and downward jumps. We further consider the case of dependence based on the FGM copula among gain sizes and gain interarrivals. Finally, in Section \ref{rv} we consider the case where the proportional parameter is uniformly distributed. 

\section{A dual risk model with causal dependence structure and proportional gains}\label{caus}
In this section, we generalize the work in \cite{boxruin}, by assuming a dependence structure among the gain size and the gain interarrival times, which affects both the gain size distribution as well as the proportional parameter. In \cite{albbox}, the authors considered the classical ruin model with causal dependencies (motivated by the work in \cite{boxperry}), where the distribution of the interclaim time depends on the actual size of the previous claim based on a (random) threshold type policy. In the following, we consider the dual model with the additional feature of the proportional gains. 

If the gain interarrival time $B_{i}$ is larger than a threshold $T_{i}$, the, the capital jumps up to the level $(1+a_{0})u+C_{i}^{(0)}$, where $C_{i}^{(0)}$ follows a hyperexponential distribution with cumulative distribution function (c.d.f.) $C_{0}(x)=\sum_{i=1}^{K}q_{i}(1-e^{-\mu_{i}x})$. Otherwise, the capital jumps up to the level $(1+a_{1})u+C_{i}^{(1)}$, where $C_{i}^{(1)}$ follows a hyperexponential distribution with c.d.f. $C_{1}(x)=\sum_{i=1}^{L}h_{i}(1-e^{-\nu_{i}x})$. The thresholds $T_{m}$ are assumed to be independent and identically distributed (i.i.d.) random variables with c.d.f. $T(.)$. We assume that $B_{i}$ are i.i.d. random variables having c.d.f. $B(.)$, density $b(.)$, and Laplace-Stieltjes transform $\phi(.)$.

In this work, we generalize the model in \cite{boxruin} by incorporating a causal dependence structure, namely that the distribution of the gain size depends on the gain interarrival time. This type of dependence it also affects the proportional parameter, which in turn affects the size of the capital jump. Furthermore, contrary to the case in \cite{boxruin} where the authors assumed exponentially distributed gain sizes (they just mentioned in \cite[Remark 2.4]{boxruin} that the model can be generalized to the case of hyperexponentially distributed gain sizes), we consider in Remark \ref{rem3} the case where the gain sizes follow a mixed Erlang distribution, a class of the phase-type distributions that can be used to approximate any given continuous distribution in $[0,\infty)$. 

\subsection{The ruin probability}
We focus on deriving the Laplace transform of the ruin probability $R(x)$ when starting in $x$, by distinguishing the two cases in which no jump up occurs before $x$ (hence ruin occurs at time $x$) and in which a jump up occurs at some time $t\in(0,x)$. Then,
\begin{equation}
    \begin{array}{rl}
         R(x)=&1-B(x)+\int_{t=0}^{x}\int_{y=0}^{\infty}\left(P(T_{m}<t)R((1+a_{1})(x-t)+y)\sum_{i=1}^{L}h_{i}\nu_{i}e^{-\nu_{i}y}\right.  \vspace{2mm}\\
         &\left.+ P(T_{m}\geq t)R((1+a_{0})(x-t)+y)\sum_{i=1}^{K}q_{i}\mu_{i}e^{-\mu_{i}y}\right)dydB(t).
    \end{array}\label{eq1}
\end{equation}
Denote the Laplace transform
\begin{displaymath}
    \rho(s):=\int_{x=0}^{\infty}e^{-sx}R(x)dx.
\end{displaymath}
Then, \eqref{eq1} becomes
\begin{equation}
    \rho(s)=\frac{1-\phi(s)}{s}+I_{0}(s)+I_{1}(s),
    \label{eq2}
\end{equation}
where
\begin{displaymath}
    \begin{array}{rl}
         I_{0}(s):=&\int_{x=0}^{\infty}e^{-sx}\int_{t=0}^{x}\int_{z=(1+a_{0})(x-t)}^{\infty}P(T_{m}\geq t) R(z)\sum_{i=1}^{K}q_{i}\mu_{i}e^{-\mu_{i}(z-(1+a_{0})(x-t))}dzdB(t)dx\vspace{2mm} \vspace{2mm}\\
         =& \int_{t=0}^{\infty}e^{-st}P(T_{m}\geq t) \sum_{i=1}^{K}q_{i}\int_{x=t}^{\infty}e^{-(x-t)[s-\mu_{i}(1+a_{0})]}\int_{z=(1+a_{0})(x-t)}^{\infty}R(z)\mu_{i}e^{-\mu_{i}z}dzdB(t)dx \vspace{2mm}\\
         =&\chi_{0}(s)\sum_{i=1}^{K}q_{i}\mu_{i}\int_{w=0}^{\infty}e^{-w[s-\mu_{i}(1+a_{0})]}\int_{z=(1+a_{0})w}^{\infty}R(z)e^{-\mu_{i}z}dzdw \vspace{2mm}\\
         =&\chi_{0}(s)\sum_{i=1}^{K}q_{i}\mu_{i}\int_{z=0}^{\infty}R(z)e^{-\mu_{i}z}\left(\frac{e^{-\frac{z}{1+a_{0}}(s-\mu_{i}(1+a_{0}))}-1}{\mu_{i}(1+a_{0})-s}\right)dz\vspace{2mm}\\
         =&\chi_{0}(s)\frac{\sum_{i=1}^{K}q_{i}\mu_{i}}{\mu_{i}(1+a_{0})-s}[\rho(\frac{s}{1+a_{0}})-\rho(\mu_{i})],
    \end{array}
\end{displaymath}
where $\chi_{0}(s):=E(e^{-sB}1(T\geq B))=\int_{x=0}^{\infty}e^{-sx}(1-T(x))dB(x)$.

Similarly,
\begin{displaymath}
 I_{1}(s):=\chi_{1}(s)\frac{\sum_{i=1}^{M}h_{i}\nu_{i}}{\nu_{i}(1+a_{1})-s}[\rho(\frac{s}{1+a_{1}})-\rho(\nu_{i})],
\end{displaymath}
where now $\chi_{1}(s):=E(e^{-sB}1(T<B))=\int_{x=0}^{\infty}e^{-sx}T(x)dB(x)$. Note that $\chi_{0}(s)+\chi_{1}(s)=\phi(s)$.

By setting for $i=0,1,$ $a_{i}(s):= \frac{s}{1+a_{i}}$, $         \widehat{\chi}_{i}(s):=\frac{\chi_{i}(s)}{1+a_{i}}$. $\psi_{i}(s):=\frac{\mu_{i}}{\mu_{i}-s}$, $\omega_{i}(s):=\frac{\nu_{i}}{\nu_{i}-s}$ \eqref{eq2} becomes
\begin{equation}
         \rho(s)=\rho(a_{0}(s))\widehat{\chi}_{0}(s)\sum_{i=1}^{K}q_{i}\psi_{i}(a_{0}(s))+\rho(a_{1}(s))\widehat{\chi}_{1}(s)\sum_{i=1}^{M}h_{i}\omega_{i}(a_{1}(s))+K(s),
\label{eq3}
\end{equation}
where
\begin{displaymath}
    K(s):=\frac{1-\phi(s)}{s}-\widehat{\chi}_{0}(s)\sum_{i=1}^{K}q_{i}\psi_{i}(a_{0}(s))\rho(\mu_{i})-\widehat{\chi}_{1}(s)\sum_{i=1}^{M}h_{i}\omega_{i}(a_{1}(s))\rho(\nu_{i}).
\end{displaymath}

Note that $a_{i}(a_{j}(s))=\frac{s}{(1+a_{0})(1+a_{1})}=a_{j}(a_{i}(s))$, thus the mappings $a_{i}(s)$ commute. Moreover, $|a_{i}(s)-a_{j}(u)|\leq k|s-u|$, where $k:=max\{1/(1+a_{0}),1/(1+a_{1})\}$, so that the mappings $a_{i}(s)$ are contraction mappings. 

After $n-1$ iterations of \eqref{eq3} we have
\begin{equation}
    \rho(s)=\sum_{l=0}^{n}L_{l,n-l}(s)\rho(a_{l,n-l}(s))+\sum_{j=0}^{n-1}\sum_{l=0}^{j}L_{l,j-l}(s)K(a_{l,j-l}(s)),\label{eq4}
\end{equation}
where $a_{l,j-l}(s):=a_{0}^{l}(a_{1}^{l-s}(s))$, with $a_{0,0}(s)=s$ and $a_{i}^{m}(s)$ is the $m$th iterate of $a_{i}(s)$, $i=0,1$, $m=0,1,2,\ldots$. Moreover, the functions $L_{l,j-l}(s)$ are computed recursively, with $L_{0,0}(s)=1$, $L_{1,0}(s):=\widehat{\chi}_{0}(s)\sum_{i=1}^{K}q_{i}\psi_{i}(a_{0}(s))$, $L_{1,0}(s):=\widehat{\chi}_{1}(s)\sum_{i=1}^{M}h_{i}\omega_{i}(a_{1}(s))$, and
\begin{equation}
    \begin{array}{rl}
       L_{l+1,j-l}(s)=  &  L_{l,j-l}(s) L_{1,0}(a_{l,j-l}(s))+ L_{l+1,j-l-1}(s) L_{0,1}(a_{l+1,j-l-1}(s)),\,j-l\geq l+1,\vspace{2mm} \\
       L_{l,j-l+1}(s)=  &  L_{l,j-l}(s) L_{0,1}(a_{l,j-l}(s))+ L_{l-1,j-l+1}(s) L_{1,0}(a_{l-1,j-l+1}(s)),\,j-l\leq l-1.
    \end{array}\label{assi}
\end{equation}
By observing that $max\{1/(1+a_{0}),1/(1+a_{1})\}<1$, it is seen following the lines in \cite[section 2]{adan} that $\rho(a_{l,j-l})=\rho(\frac{s}{(1+a_{0})^{l}(1+a_{1})^{j-l}})$ converges geometrically fast to $\rho(0)=1$. Moreover,
\begin{displaymath}
    |L_{l,j-l}(s)|\leq \binom{j}{l,j-l},
\end{displaymath}
and hence, the sums in \eqref{eq4} are bounded by one. On the other hand, $K(a_{l,j-l}(s))$ for large values of $j$ converges to the constant 
\begin{displaymath}
    \Bar{b}-\frac{\chi_{0}(0)}{1+a_{0}}\sum_{i=1}^{K}q_{i}\rho(\mu_{i})-\frac{\chi_{1}(0)}{1+a_{1}}\sum_{i=1}^{M}h_{i}\rho(\nu_{i})<\infty.
\end{displaymath}
Hence,
\begin{equation}
    \rho(s)=\lim_{n\to\infty}\sum_{l=0}^{n}L_{l,n-l}(s)+\sum_{j=0}^{\infty}\sum_{l=0}^{j}L_{l,j-l}(s)K(a_{l,j-l}(s)).\label{eq5}
\end{equation}
We still need to obtain the values of $\rho(\mu_{i})$, $i=1,\ldots,K$, and $\rho(\nu_{i})$, $i=1,\ldots,L$. Setting $s=\mu_{i}$, $i=1,\ldots,K$, and $s=\nu_{i}$, $i=1,\ldots,L$, we can obtain a set of equations to derive these unknowns.   
\begin{remark}\label{rem3}
    We now consider the case where $C_{i}^{(k)}$, $k=0,1,$ are i.i.d. random variables that follow a mixed Erlang distribution. More precisely, we assume that when $B_{i}\geq T_{i}$, then the claim sizes $C_{i}^{(0)}$ have c.d.f.
    \begin{displaymath}
        C_{0}(x)=\sum_{n_{0}=1}^{N_{0}}k_{n_{0}}(1-e^{-\mu_{0}x}\sum_{l=0}^{n_{0}-1}\frac{(\mu_{0}x)^{l}}{l!}),\,x\geq 0,
    \end{displaymath}
where $\sum_{n_{0}=1}^{N_{0}}k_{n_{0}}=1$. Similarly, when $B_{i}<T_{i}$, then the claim sizes $C_{i}^{(1)}$ have c.d.f.
    \begin{displaymath}
        C_{1}(x)=\sum_{n_{1}=1}^{N_{0}}k_{n_{1}}(1-e^{-\mu_{1}x}\sum_{l=0}^{n_{1}-1}\frac{(\mu_{1}x)^{l}}{l!}),\,x\geq 0.
    \end{displaymath}
    with $\sum_{n_{1}=1}^{N_{1}}k_{n_{1}}=1$. In other words, we assume that the claim sizes follow with
probability $k_{n_{m}}$, $n_{m} = 1,\ldots,N_{m}$, $m=0,1,$ an Erlang distribution of scale parameter $\mu_{m}$ and $n_{m}$ stages.

The class of the phase-type distributions of the above form is dense in the space of distribution functions defined on $[0,\infty)$, and thus, for any distribution
function $B$, there is a sequence $B_{n}$ of phase-type distributions of this class that
converges weakly to $B$ as $n$ goes to infinity; see \cite{schass}. This class, since it may be used to approximate
any given continuous distribution on $[0,\infty)$ arbitrarily close. 

The whole analysis can be repeated, although there will be some difficulties. More precisely, by applying the same steps \eqref{eq2} is still valid but now,
\begin{displaymath}
    \begin{array}{rl}
       I_{0}(s)=  &\chi_{0}(s)\sum_{n_{0}=1}^{N_{0}}k_{n_{0}}\mu_{n_{0}}^{n_{0}} \int_{w=0}^{\infty}e^{-w[s-\mu_{n_{0}}(1+a_{0})]}\int_{z=(1+a_{0})w}^{\infty}R(z)e^{-\mu_{0}z}(z-(1+a_{0})w)^{n_{0}-1}dzdw \vspace{2mm} \\
        = & \chi_{0}(s)\sum_{n_{0}=1}^{N_{0}}k_{n_{0}}\mu_{n_{0}}^{n_{0}}\sum_{l=0}^{n_{0}-1}\frac{(-1)^{n_{0}-1-l}(1+a_{0})^{n_{0}-1-l}}{l!(n_{0}-1-l)!}\int_{z=0}^{\infty}z^{l}e^{-\mu_{n_{0}}z}R(z)\int_{w=0}^{\frac{z}{1+a_{0}}}w^{n_{0}-1-l}e^{-w[s-\mu_{n_{0}}(1+a_{0})]}dwdz\vspace{2mm}\\
        =& \chi_{0}(s)\sum_{n_{0}=1}^{N_{0}}k_{n_{0}}\mu_{n_{0}}^{n_{0}}\sum_{l=0}^{n_{0}-1}\frac{(-1)^{n_{0}-1-l}(1+a_{0})^{n_{0}-1-l}}{l!(s-\mu_{n_{0}}(1+a_{0}))^{n-l}}\int_{z=0}^{\infty}z^{l}e^{-\mu_{0}z}R(z)\vspace{2mm}\\&\times\left[1-e^{-\frac{z}{1+a_{0}}(s-\mu_{n_{0}}(1+a_{0}))}\sum_{i=0}^{n_{0}-1-l}\frac{(\frac{z}{1+a_{0}}(s-\mu_{n_{0}}(1+a_{0})))^{i}}{i!}\right]dz\vspace{2mm}\\
        =&\chi_{0}(s)\sum_{n_{0}=1}^{N_{0}}k_{n_{0}}(-1)^{n_{0}-1}\sum_{l=0}^{n_{0}-1}\left(\frac{\mu_{n_{0}}}{s-(1+a_{0})\mu_{n_{0}}}\right)^{n_{0}-l}\frac{(1+a_{0})^{n_{0}-1-l}}{l!}\rho^{(l)}(\mu_{n_{0}})\vspace{2mm}\\
        &-\rho(\frac{s}{1+a_{0}})\chi_{0}(s)\sum_{n_{0}=1}^{N_{0}}k_{n_{0}}(1+a_{0})^{n_{0}-1}(-1)^{n_{0}}\left(\frac{\mu_{n_{0}}}{s-(1+a_{0})\mu_{n_{0}})}\right)^{n_{0}},
    \end{array}
\end{displaymath}
where $\rho^{(l)}(\mu_{n_{0}})$ denotes the $l$th derivative of $\rho(s)$ at point $s=\mu_{n_{0}}$, $n_{0}=1,\ldots,N_{0}$.
Similarly,
\begin{displaymath}
    \begin{array}{rl}
        I_{1}(s)= &\chi_{1}(s)\sum_{n_{1}=1}^{N_{1}}k_{n_{1}}(-1)^{n_{1}-1}\sum_{l=0}^{n_{1}-1}\left(\frac{\mu_{n_{1}}}{s-(1+a_{0})\mu_{n_{1}}}\right)^{n_{1}-l}\frac{(1+a_{1})^{n_{1}-1-l}}{l!}\rho^{(l)}(\mu_{n_{1}})\vspace{2mm}\\
        &-\rho(\frac{s}{1+a_{1}})\chi_{1}(s)\sum_{n_{1}=1}^{N_{1}}k_{n_{1}}(1+a_{1})^{n_{1}-1}(-1)^{n_{1}}\left(\frac{\mu_{n_{1}}}{s-(1+a_{1})\mu_{n_{1}})}\right)^{n_{1}},
    \end{array}
\end{displaymath}
where $\rho^{(l)}(\mu_{n_{1}})$ denotes the $l$th derivative of $\rho(s)$ at point $s=\mu_{n_{1}}$, $n_{1}=1,\ldots,N_{1}$. Therefore, we come up with the following functional equation:
\begin{equation}
    \rho(s)=\rho(a_{0}(s))\widehat{\chi}_{0}(s)\sum_{n_{0}=1}^{N_{0}}k_{n_{0}}\left(\frac{\mu_{n_{0}}}{\mu_{n_{0}}+a_{0}(s)}\right)^{n_{0}}+\rho(a_{1}(s))\widehat{\chi}_{1}(s)\sum_{n_{1}=1}^{N_{1}}k_{n_{1}}\left(\frac{\mu_{n_{1}}}{\mu_{n_{1}}+a_{1}(s)}\right)^{n_{1}}+K(s),
\label{eq33}
\end{equation}
where now,
\begin{displaymath}
\begin{array}{rl}
    K(s):=&\frac{1-\phi(s)}{s}-\widehat{\chi}_{0}(s)\sum_{n_{0}=1}^{N_{0}}k_{n_{0}}\sum_{l=0}^{n_{0}-1}\frac{(-1)^{l}}{l!}\left(\frac{\mu_{n_{0}}}{\mu_{n_{0}}+a_{0}(s)}\right)^{n_{0}-l}\rho^{(l)}(\mu_{n_{0}})\vspace{2mm}\\
    &-\widehat{\chi}_{1}(s)\sum_{n_{1}=1}^{N_{1}}k_{n_{1}}\sum_{l=0}^{n_{1}-1}\frac{(-1)^{l}}{l!}\left(\frac{\mu_{n_{1}}}{\mu_{n_{1}}+a_{1}(s)}\right)^{n_{1}-l}\rho^{(l)}(\mu_{n_{1}}).\end{array}
\end{displaymath}
The form of the equation \eqref{eq33} is the same as the one in \eqref{eq3}, so we can apply the same iterative procedure to obtain $\rho(s)$. Note however, that we need to obtain the terms $\rho^{(d_{k})}(\mu_{n_{k}})$, $k=0,1$, $d_{k}=0,1,\ldots,N_{k}-1$, i.e., the values of the $d_{k}$th derivative of $\rho(s)$ ate points $s=\mu_{n_{k}}$, $k=0,1.$ This task can be accomplished similarly as in the 
\end{remark}
\subsection{The time to ruin}\label{time}
We now turn our attention to $\tau_{x}$, i.e., the time to ruin starting at level $x$. Then,
\begin{displaymath}
\begin{array}{rl}
    E(e^{\alpha \tau_{x}})=&e^{-ax}(1-B(x))+\int_{t=0}^{x}e^{-\alpha t}\int_{y=0}^{\infty}\left(P(T<t) E(e^{\alpha \tau_{(1+a_{0})(x-t)+y}})\sum_{i=1}^{L}h_{i}\nu_{i}e^{-\nu_{i}y}\right.\vspace{2mm}\\
    &\left.+P(T\geq t) E(e^{\alpha \tau_{(1+a_{1})(x-t)+y}})\sum_{j=1}^{K}q_{j}\mu_{j}e^{-\mu_{j}y}\right)dydF(t).
    \end{array}
\end{displaymath}
\begin{remark}
    Note that by taking $\alpha=0$ in the expression above, we obtain the expression for the ruin probability $R(x)$ given in \eqref{eq1}.
\end{remark}
Letting,
\begin{displaymath}
    \tau(s,\alpha):=\int_{x=0}^{\infty}e^{-sx}E(e^{\alpha \tau_{x}})dx,
\end{displaymath}
we obtain after some algebra:
\begin{equation}
    \tau(s,\alpha)=\tau(a_{0}(s),\alpha)\widehat{\chi}_{0}(s+\alpha)\sum_{j=1}^{K}q_{j}\psi_{j}(a_{0}(s))+\tau(a_{1}(s),\alpha)\widehat{\chi}_{1}(s+\alpha)\sum_{i=1}^{L}h_{i}\omega_{i}(a_{1}(s))+\Tilde{K}(s,\alpha),
    \label{q01}
\end{equation}
where
\begin{displaymath}
    \Tilde{K}(s,\alpha)=\frac{1-\phi(s+\alpha)}{s+\alpha}-\widehat{\chi}_{0}(s+\alpha)\sum_{j=1}^{K}q_{j}\psi_{j}(a_{0}(s))\tau(\mu_{j},\alpha)-\widehat{\chi}_{1}(s+\alpha)\sum_{i=1}^{L}h_{i}\omega_{i}(a_{1}(s))\tau(\nu_{i},\alpha).
\end{displaymath}
Then, \eqref{q01} is rewritten as 
\begin{equation}
    \tau(s,\alpha)=\tau(a_{0}(s),\alpha)H_{0}(s,\alpha)+\tau(a_{1}(s),\alpha)H_{1}(s,\alpha)+\Tilde{K}(s,\alpha),\label{ghu} 
\end{equation}
where,
\begin{displaymath}
    H_{0}(s,\alpha)=\widehat{\chi}_{0}(s+\alpha)\sum_{j=1}^{K}q_{j}\psi_{j}(a_{0}(s)),\,\,H_{1}(s,\alpha)=\widehat{\chi}_{1}(s+\alpha)\sum_{i=1}^{L}h_{i}\omega_{i}(a_{1}(s)).
\end{displaymath}
After $n-1$ iterations, \eqref{ghu} yields,
\begin{equation}
   \tau(s,\alpha)=\sum_{l=0}^{n}\tilde{L}_{l,n-l}(s,a)\tau(a_{l,n-l}(s),\alpha)+\sum_{j=0}^{n-1}\sum_{l=0}^{j}\tilde{L}_{l,j-l}(s,\alpha)\tilde{K}(a_{l,j-l}(s),\alpha),\label{eq02}
\end{equation}
where now the functions $\tilde{L}_{l,j-l}(s,\alpha)$ are computed recursively, with $\tilde{L}_{0,0}(s,\alpha)=1$, $\tilde{L}_{1,0}(s,\alpha):=\widehat{\chi}_{0}(s+\alpha)\sum_{j=1}^{K}q_{j}\psi_{i}(a_{0}(s))$, $\tilde{L}_{1,0}(s,\alpha):=\widehat{\chi}_{1}(s+\alpha)\sum_{i=1}^{M}h_{i}\omega_{i}(a_{1}(s))$, and
\begin{displaymath}
    \begin{array}{rl}
       \tilde{L}_{l+1,j-l}(s,\alpha)=  &  \tilde{L}_{l,j-l}(s,\alpha) \tilde{L}_{1,0}(a_{l,j-l}(s),\alpha)+ \tilde{L}_{l+1,j-l-1}(s,\alpha) \tilde{L}_{0,1}(a_{l+1,j-l-1}(s),\alpha),\,j-l\geq l+1,\vspace{2mm} \\
       \tilde{L}_{l,j-l+1}(s,\alpha)=  &  \tilde{L}_{l,j-l}(s,\alpha) \tilde{L}_{0,1}(a_{l,j-l}(s),\alpha)+ \tilde{L}_{l-1,j-l+1}(s,\alpha) \tilde{L}_{1,0}(a_{l-1,j-l+1}(s),\alpha),\,j-l\leq l-1.
    \end{array}
\end{displaymath}
For large $j$, $\tilde{K}(a_{l,j-l}(s),\alpha)$ approaches a function of $\alpha$. Following similar arguments as above yields,
\begin{equation}
    \tau(s,\alpha)=\lim_{n\to\infty}\sum_{l=0}^{n}\tilde{L}_{l,n-l}(s,\alpha)+\sum_{j=0}^{\infty}\sum_{l=0}^{j}\tilde{L}_{l,j-l}(s,\alpha)\tilde{K}(a_{l,j-l}(s),\alpha).\label{eq05}
\end{equation}
The values of $\tau(\mu_{j})$, $j=1,\ldots,K$, $\tau(\nu_{i})$, $i=1,\ldots,L$, can be obtained by solving a system of $K+L$ equations, which is derived by substituting $s=\mu_{j}$, $j=1,\ldots,K$, and $s=\nu_{i}$, $i=1,\ldots,L$, in \eqref{eq05}.
\begin{remark}
    Similar arguments can be used to cope with the case where $C_{i}^{k}$, $k=0,1,$ are i.i.d. random variables that follow a mixed Erlang distribution, although it would be trickier due to the presence of additional unknown terms that corresponds to the derivatives of $\tau(s,a)$ at specific points (see Remark \ref{rem3}), and further details are omitted.  
\end{remark}

\section{The dual risk model with additive/proportional gains and dependencies based on FGM copula}\label{copu}
We assume that gain interarrival times $B_{i}$, and the gain sizes $C_{i}$ are dependent  based on the
Farlie-Gumbel-Morgenstern (FGM) copula. As indicated in \cite{nelsen}, the FGM copula is a perturbation of the product copula and a first order approximation to the
Ali Mikhail Haq, Frank and Placket copulas. A dependent structure based on copulas has been considered so far in the classical risk reserve process, but to our best knowledge there is a lack of results regarding the use of copulas on the dual risk model. In this work we fill this gap, by considering copulas to describe the dependence structure in this general dual risk model with proportional gains. 

In \cite{cossette1}, the authors considered an extension of the classical compound Poisson risk model where the claim  amounts  and  claim  inter-arrival  times  are  dependent through a FGM copula. In their work, they assumed that interclaim arrival times are exponentially distributed. Later, in \cite{chatzi}, the authors generalized the work in \cite{cossette1}, by considering  an extension to the renewal or Sparre Andersen risk process, where the claim interarrivals are Erlang distributed. Later, in \cite{cossete}, the authors considered an extension to the classical compound Poisson risk model in which, the dependence structure between the claim amounts and the interclaim time is embedded via a generalized FGM (GFGM) copula. We also mentioned the recent work in \cite{boxman}, where some risk models with a semi-linear dependent structure were discussed.

 As stated in \cite{nelsen}, copulas are functions which "join or "couple" multivariate distribution functions to their one-dimensional marginal distribution functions". In other words, a copula itself is a multivariate distribution function whose inputs are the
respective marginal cumulative probability distribution functions for the random variables
of interest. A bivariate copula $C$ is a joint distribution function on $[0,1]\times[0,1]$ with standard uniform marginal  distributions.  By  the  theorem  of  Sklar  (see  e.g. \cite{nelsen}),  any  bivariate distribution   function $F$ with marginals $F_1$ and $F_2$ can   be   written   as $F(x,y)=C(F_{1}(x),F_{2}(y))$, for some copula $C$. This copula is unique if $F$ is continuous. Otherwise, it is uniquely defined on the range of the marginals. We refer the reader to \cite{joe,nelsen} for further details on copulas. Modeling the dependence structure between random variables (r.v.) using copulas has become popular in actuarial science and financial risk management. The  reader  may refer to e.g. \cite{denoi,mcneil} for  applications  of  copulas  in  actuarial science  and  financial  risk  management. We also mentioned \cite{albteu} that used copulas to define the joint distribution for the interclaim time r.v. and the claim amount r.v..

In this section, we generalize the model in \cite{boxruin}, by considering a dependence structure among interarrival gains, say $B_{i}$, and gain sizes $C_{i}$ based on the FGM copula. We assume that $B_{i}$ are i.i.d. random variables having c.d.f. $B(.)$, density $b(.)$, and Laplace-Stieltjes transform $\phi(.)$, and assuming $C_{i}$ to follow Erlang distribution. In particular, we assume that $\{(B_{i},C_{i}),i\geq 1\}$ form a sequence of i.i.d. random vectors distributed as the canonical r.v. $(B,C)$, in which the components are dependent. The joint p.d.f. of $(B,C)$ is denoted by $f_{B,C}(x,t)$, $x,t\geq 0$, the c.d.f. by $F_{B,C}(x,t)$. The FGM copula is used to define the joint
distribution of $(B,C)$. The FGM copula is defined by
\begin{displaymath}
    C_{\theta}^{FGM}(u,v):=uv+\theta uv(1-u)(1-v),
\end{displaymath}
for $(u,v)\in[0,1]^{2}$, and $\theta\in(-1,1)$. The FGM copula allows positive and negative dependence, and it further includes the independence copula for $\theta=0$. Its tractability, and simplicity, makes FGM copula quite useful in applications to describe dependence structures. 

The pdf associated to the above expression is $c_{\theta}^{FGM}(u,v):=\frac{\partial^{2}}{\partial u\partial v}C_{\theta}^{FGM}(u,v)=1+\theta(1-2u)(1-2v)$, and then, the bivariate pdf of $(B,C)$ is given by 
\begin{displaymath}
\begin{array}{rl}
     f_{B,C}(x,t)=&c_{\theta}^{FGM}(F_{B}(x),F_{C}(t))f_{B}(x)f_{C}(t)= f_{B}(x)f_{C}(t)+\theta h(x)(2\bar{F}_{C}(t)-1),
\end{array}
\end{displaymath}
with $h(x):=f_{B}(x)(1-2F_{B}(x))$ with Laplace transform $h^{*}(s)=\int_{0}^{\infty}e^{-sx}h(x)dx$, and $\bar{F}_{C}(t):=1-F_{C}(t)$. In our case,
\begin{equation}
    \begin{array}{rl}
         f_{B,C}(x,t)=&f_{B}(x)\frac{\mu^{n}}{(n-1)!}t^{n-1}e^{-\mu t}\\           &+\theta h(x)\left[2\frac{\mu^{n}}{(n-1)!}t^{n-1}e^{-2\mu t}(\sum_{i=0}^{n-1}\frac{(\mu t)^{i}}{i!})-\frac{\mu^{n}}{(n-1)!}t^{n-1}e^{-\mu t}\right].
    \end{array}\label{biv1}   
\end{equation}
Therefore,
\begin{equation}
    \begin{array}{rl}
         R(x)=&1-B(x)+\int_{t=0}^{x}\int_{y=0}^{\infty}R((1+a)(x-t)+y)f_{B,C}(t,y)dydt.
    \end{array}\label{eq11}
\end{equation}
Taking Laplace transform, we come up with the following expression
\begin{displaymath}
    \rho(s)=\frac{1-\phi(s)}{s}+\phi(s)I_{0}(s)+I_{1}(s)-I_{2}(s),
\end{displaymath}
where
\begin{displaymath}
    \begin{array}{rl}
         I_{0}(s)=&\phi(s)\int_{x=t}^{\infty}e^{-s(x-t)}\int_{z=(1+a)(x-t)}^{\infty}R(z)\mu^{n}\frac{(z-(1+a)(x-t))^{n-1}}{(n-1)!}e^{-\mu(z-(1+a)(x-t)}dzdx\vspace{2mm}  \\
        = &\sum_{l=0}^{n-1}\frac{(-1)^{n-1}(1+a)^{n-1-l}\mu^{n-l}}{l!(s-\mu(1+a))^{n-l}}\rho^{(l)}(\mu)+\frac{1}{1+a}\left(\frac{\mu(1+a)}{\mu(1+a)-s}\right)^{n}\rho(\frac{s}{1+a}),\vspace{2mm}\\
        I_{1}(s)=&\theta h^{*}(s)\int_{x=t}^{\infty}e^{-s(x-t)}\int_{z=(1+a)(x-t)}^{\infty}2R(z)\mu^{n}\frac{(z-(1+a)(x-t))^{n-1}}{(n-1)!}e^{-2\mu(z-(1+a)(x-t))}\sum_{i=0}^{n-1}\frac{(\mu(z-(1+a)(x-t)))^{i}}{i!}dzdx\vspace{2mm}\\
        =&2\theta h^{*}(s)\sum_{i=0}^{n-1}\binom{n-1+i}{i}\left[\sum_{l=0}^{n-1+i}\frac{(-1)^{n+i-1}}{l!2^{l}(1+a)}\left(\frac{\mu(1+a)}{s-2\mu(1+a)}\right)^{n+i}\rho^{(l)}(2\mu)\right.\vspace{2mm}\\
        &\left.+\frac{1}{1+a}\left(\frac{\mu(1+a)}{2\mu(1+a)-s}\right)^{n+i}\rho(\frac{s}{1+a})\right],\vspace{2mm}\\
        I_{2}(s)=&\theta h^{*}(s)I_{0}(s).
    \end{array}
\end{displaymath}
Combining the above, we come up with the following functional equation:
\begin{equation}
    \rho(s)=J(s;\theta)\rho(\frac{s}{1+a})+H(s;\theta),\label{eqo}
\end{equation}
where,
\begin{equation}
    \begin{array}{rl}
        J(s;\theta)=&\frac{\phi(s)-\theta h^{*}(s)}{1+a}\left(\frac{\mu(1+a)}{\mu(1+a)-s}\right)^{n}+ \frac{2\theta h^{*}(s)}{1+a}\sum_{i=0}^{n-1}\binom{n-1+i}{i}\left(\frac{\mu(1+a)}{2\mu(1+a)-s}\right)^{n+i},\vspace{2mm}\\
        H(s;\theta)= & \frac{1-\phi(s)}{s}+\widehat{H}(s;\theta)\\=&
        \frac{1-\phi(s)}{s}+\frac{(-1)^{n-1}(\phi(s)-\theta h^{*}(s))}{1+a}\sum_{l=0}^{n-1}\frac{\rho^{(l)}(\mu)}{l!}\left(\frac{\mu(1+a)}{s-\mu(1+a)}\right)^{n-l}\vspace{2mm}\\
        &+2\theta h^{*}(s)\sum_{i=0}^{n-1}\binom{n-1+i}{i}\sum_{l=0}^{n-1+i}\frac{(-1)^{n+i-1}}{l!2^{l}(1+a)}\left(\frac{\mu(1+a)}{s-2\mu(1+a)}\right)^{n+i}\rho^{(l)}(2\mu).
    \end{array}\label{dol}
\end{equation}
Note that \eqref{eqo} has the same form as in \cite[eq. (2.5)]{boxruin}, so we can apply a similar arguments to solve it. Iterating $N-1$ times \eqref{eqo} results in
\begin{displaymath}
    \rho(s)=\sum_{k=0}^{N-1}\prod_{j=0}^{k-1}J(\frac{s}{(1+a)^{j}}   ;\theta)H(\frac{s}{(1+a)^{k}};\theta)+\rho(\frac{s}{(1+a)^{N}})\prod_{j=0}^{N-1}J(\frac{s}{(1+a)^{j}};\theta),
\end{displaymath}
with an empty product being equal to 1. Observe that for large $k$ $J(\frac{s}{(1+a)^{k}}   ;\theta)$ approaches
\begin{displaymath}
    \begin{array}{l}
      \frac{\phi(0)-\theta h^{*}(0)}{1+a}+\frac{2\theta h^{*}(0)}{1+a}\sum_{i=0}^{n-1}\binom{n-1+i}{i}(\frac{1}{2})^{n+i}=\frac{\phi(0)-\theta h^{*}(0)}{1+a}+\frac{\theta h^{*}(0)}{1+a}(\frac{1}{2}+\frac{1}{2})^{n-1}=\frac{\phi(0)}{1+a}<1,
    \end{array}
\end{displaymath}
while $H(\frac{s}{(1+a)^{k}})$ to a constant (as $k\to\infty$). Thus,
    \begin{equation}
    \rho(s)=\sum_{k=0}^{\infty}\prod_{j=0}^{k-1}J(\frac{s}{(1+a)^{j}}   ;\theta)H(\frac{s}{(1+a)^{k}};\theta).\label{gahy}    
    \end{equation}
The terms $\rho^{(l)}(\mu)$, $l=0,1,\ldots,n-1$, $\rho^{(l)}(2\mu)$, $l=0,1,\ldots,2n-2$, are obtained as follows: First, we simply differentiating $l$ times \eqref{gahy} with respect to $s$, and setting $s=\mu$, and $s=2\mu$, respectively. Then, a system of $2n+1$ equations is derived with unknown the terms $\rho^{(l)}(\mu)$, $l=0,1,\ldots,n-1$, $\rho^{(l)}(2\mu)$, $l=0,1,\ldots,2n-2$. 
\begin{remark}
    Setting $\theta=0$, i.e., by assuming independence among the gain interarrival and the gain size, and letting $n=1$, i.e., by considering exponentially distributed gain sizes, we recover the functional equation (2.5) in \cite{boxruin}.
\end{remark}
Then, we summarize the above in the following main result:
\begin{theorem}
    The Laplace transform of the ruin probability $\rho(s)$ is given in \eqref{ghy}, where the terms $\rho^{(l)}(\mu)$, $l=0,1,\ldots,n-1$, $\rho^{(l)}(2\mu)$, $l=0,1,\ldots,2n-2$, are obtained as the solution of a system of $3n-1$ equations that is derived by differentiating $l$ times \eqref{ghy} with respect to $s$, and setting $s=\mu$, and $s=2\mu$, respectively.
\end{theorem}
\begin{remark}
    Note that both $J(s;\theta)$ and $H(s,\theta)$ have a singularity at $s=\mu(1+a)$, and $s=2\mu(1+a)$, and thus, one expects that the expression $\rho(s)$ in \eqref{gahy} has a singularity at $s=\mu(1+a)^{j+1}$, and $s=2\mu(1+a)^{j+1}$, $j=0,1,\ldots$. However, such a singularity is removable (as in the independent case that was treated in \cite[Remark 2.2]{boxruin}).
\end{remark}
\subsection{The time to ruin}
Working on the same spirit as in subsection \ref{time}, and denoting by $\tau(s,\alpha)$ the Laplace transform of $\tau_{x}$ we obtain
\begin{displaymath}
    \tau(s,\alpha)=\frac{1-\phi(s+\alpha)}{s+\alpha}+\phi(s+\alpha)P_{0}(s,\alpha,\theta)+P_{1}(s,\alpha;\theta)-P_{2}(s,\alpha;\theta),
\end{displaymath}
where now,
\begin{displaymath}
    \begin{array}{rl}
         P_{0}(s,\alpha,\theta)=&\frac{1}{1+a}\left\{\sum_{i=0}^{n-1}\frac{(-1)^{n-1}}{i!}(\frac{\mu(1+a)}{s-\mu(1+a)})^{n-i}\tau^{(i)}(\mu,\alpha)+(\frac{\mu(1+a)}{s-\mu(1+a)})^{n}\tau(\frac{s}{1+a},\alpha)\right\},\vspace{2mm} \\
      P_{1}(s,\alpha,\theta)=& 2\theta h^{*}(s+\alpha)\sum_{i=0}^{n-1}\binom{n-1+i}{i}\left[\sum_{l=0}^{n-1+i}\frac{(-1)^{n+i-1}}{l!2^{l}(1+a)}\left(\frac{\mu(1+a)}{s-2\mu(1+a)}\right)^{n+i}\tau^{(l)}(2\mu,\alpha)\right.\vspace{2mm}\\
        &\left.+\frac{1}{1+a}\left(\frac{\mu(1+a)}{2\mu(1+a)-s}\right)^{n+i}\tau(\frac{s}{1+a},\alpha)\right],\vspace{2mm}\\
        P_{2}(s,\alpha;\theta)=&\theta h^{*}(s+\alpha)P_{0}(s,\alpha;\theta).
    \end{array}
\end{displaymath}
Substituting back, we finally obtain
\begin{equation}
    \tau(s,\alpha)=\tau(\frac{s}{1+a},\alpha)\tilde{J}(s,\alpha;\theta)+\tilde{H}(s,\alpha;\theta),\label{bnv}
\end{equation}
where
\begin{displaymath}
    \begin{array}{rl}
         \tilde{J}(s,\alpha;\theta)=& \frac{\phi(s+\alpha)-\theta h^{*}(s+\alpha)}{1+a}\left(\frac{\mu(1+a)}{\mu(1+a)-s}\right)^{n}+ \frac{2\theta h^{*}(s+\alpha)}{1+a}\sum_{i=0}^{n-1}\binom{n-1+i}{i}\left(\frac{\mu(1+a)}{2\mu(1+a)-s}\right)^{n+i},\vspace{2mm} \\
        \tilde{H}(s,\alpha;\theta)= &         \frac{1-\phi(s+\alpha)}{s+\alpha}+\frac{(-1)^{n-1}(\phi(s+\alpha)-\theta h^{*}(s+\alpha))}{1+a}\sum_{l=0}^{n-1}\frac{\tau^{(l)}(\mu,\alpha)}{l!}\left(\frac{\mu(1+a)}{s-\mu(1+a)}\right)^{n-l}\vspace{2mm}\\
        &+2\theta h^{*}(s+\alpha)\sum_{i=0}^{n-1}\binom{n-1+i}{i}\sum_{l=0}^{n-1+i}\frac{(-1)^{n+i-1}}{l!2^{l}(1+a)}\left(\frac{\mu(1+a)}{s-2\mu(1+a)}\right)^{n+i}\tau^{(l)}(2\mu,\alpha).
    \end{array}
\end{displaymath}

\begin{remark}
    Consider the case where when the $i$th jump upwards occur (while the surplus just before is $u$), the jump size is $au+C_{i}$ with probability $p$, and $D_{i}$ with probability $q:=1-p$. However, there is a dependence based on FGM copula. In particular, we assume that $C_{i}$ (resp. $D_{i}$) is Erlang distributed with parameters $n$, $\mu$ (resp. $m$, $\nu$). The joint bivariate pdf of $(B,C)$, is given by \eqref{biv1} (where $\theta\equiv\theta_{1}$), while the one of $(B,D)$ by
    \begin{equation}
    \begin{array}{rl}
         f_{B,D}(x,t)=&f_{B}(x)\frac{\nu^{n}}{(m-1)!}t^{m-1}e^{-\nu t}\\           &+\theta_{2} h(x)\left[2\frac{\nu^{n}}{(m-1)!}t^{m-1}e^{-2\nu t}(\sum_{i=0}^{m-1}\frac{(\nu t)^{i}}{i!})-\frac{\nu^{n}}{(m-1)!}t^{m-1}e^{-\nu t}\right],
    \end{array}\label{biv2}   
\end{equation}
with $\theta_{2}\in(-1,1)$. Then,
\begin{equation}
    \begin{array}{rl}
         R(x)=&1-B(x)+p\int_{t=0}^{x}\int_{y=0}^{\infty}R((1+a)(x-t)+y)f_{B,C}(t,y)dydt,\vspace{2mm}\\
         &+q\int_{t=0}^{x}\int_{y=0}^{\infty}R(x-t+y)f_{B,D}(t,y)dydt.
    \end{array}\label{eq1v1}
\end{equation}
By taking $p=1$, we get the previous model, while for $a=0$, we get the classical dual risk model with dependence based on the FGM copula. Using similar calculations as those that leads to \eqref{eqo}, we obtain after heavy but straightforward calculations
\begin{equation}
    \rho(s)[1-q\tilde{J}(s;\theta_{2})]=\frac{1-\phi(s)}{s}+p(J(s;\theta_{1})\rho(\frac{s}{1+a})+\widehat{H}(s;\theta_{1}))+q\tilde{H}(s;\theta_{2}),\label{fun1}
\end{equation}
where $J(s;\theta_{1})$, $H(s,\theta_{1})$ as given in \eqref{dol} for $\theta=\theta_{1}$, and
\begin{equation}
    \begin{array}{rl}
        \tilde{J}(s;\theta_{2})= &(\phi(s)-\theta_{2}h^{*}(s))(\frac{\nu}{\nu-s})^{m}+2\theta_{2} h^{*}(s)\sum_{i=0}^{m-1}\binom{m-1+i}{i} (\frac{\nu}{2\nu-s})^{m+i},\vspace{2mm} \\
         \tilde{H}(s;\theta_{2})= & (\phi(s)-\theta_{2}h^{*}(s))(-1)^{m-1}\sum_{i=0}^{m-1}(\frac{\nu}{s-\nu})^{m-i}\frac{\rho^{(i)}(\nu)}{i!}\vspace{2mm}\\
         &+2\theta_{2}h^{*}(s)\sum_{i=0}^{m-1}\binom{m-1+i}{i} \sum_{l=0}^{m-1+i}(\frac{\nu}{s-2\nu})^{m+i-l}\frac{(-1)^{m+i-l}}{l!2^{l}}\rho^{(l)}(2\nu).
    \end{array}\label{eee}
\end{equation}
After simple algebraic arguments, \eqref{fun1} is rewritten as, 
\begin{equation}
    \rho(s)=J_{1}(s;\theta_{1},\theta_{2})\rho(\frac{s}{1+a})+H_{1}(s;\theta_{1},\theta_{2}),\label{fun}
\end{equation}
where,
\begin{displaymath}
    \begin{array}{rl}
         J_{1}(s;\theta_{1},\theta_{2})=&\frac{p J(s;\theta_{1})}{1-q\tilde{J}(s;\theta_{2})},\vspace{2mm}  \\
        H_{1}(s;\theta_{1},\theta_{2})= &\frac{\frac{1-\phi(s)}{s}+p\widehat{H}(s;\theta_{1})+q\tilde{H}(s;\theta_{2})}{1-q\tilde{J}(s;\theta_{2})}.
    \end{array}
\end{displaymath}
To proceed, we need to investigate the number of zeroes of $1-q\tilde{J}(s;\theta_{2})$ in the right-half complex plane.
\begin{proposition}\label{propo}
    For $\theta_{2}\neq 0$, the equation $q\tilde{J}(s;\theta_{2})=1$ has exactly $3m-1$ roots, say $s_{1},s_{2},\ldots,s_{3m-1}$ in the right-half complex plane, i.e., $Re(s_{j})>0$, $j=1,2,\ldots,3m-1$.
\end{proposition}
\begin{proof}
    The proof is based on Rouch\'e's theorem \cite{tit}. Since $q\tilde{J}(s;\theta_{2})=1$ can be rewritten as 
    \begin{equation}
        \begin{array}{r}
            q\left[\nu^{m}\phi(s)(2\nu-s)^{2m-1}+\theta\nu^{m}h^{*}(s)\left(2\sum_{i=0}^{m-1}\binom{m+i-1}{i}\nu^{i}(\nu-s)^{m}(2\nu-s)^{m-i-1}-(2\nu-s)^{2m-1}\right)\right]\vspace{2mm}\\
=(\nu-s)^{m}(2\nu-s)^{2m-1},
        \end{array}\label{bhjx}
    \end{equation}
    it suffices to show that \eqref{bhjx} has exactly $3m-1$ roots with positive real parts. Let $r>0$ be a sufficiently large number, and denote by $C_{r}$ the contour containing the imaginary axis running from $ir$ to $-ir$ and a semicircle with radius $r$ running
clockwise from $ir$ to $-ir$, i.e., $C_{r}=\{s\in\mathbb{C}: |s|=r, Re(s)\geq 0, r>0\text{ fixed}\}$. Denote by $C$ the limiting contour, by letting $r\to\infty$. We distinguish two cases according to as $Re(s)>0$ or $Re(s)=0$.

For $s$, such that $Re(s)>0$, we have that $|2\nu-s|\to\infty$, and $|\nu-s|\to\infty$ as $r\to\infty$. Thus,
\begin{displaymath}
    \begin{array}{rl}
        |q\tilde{J}(s;\theta_{2})|\leq &q\left[|\phi(s)|\frac{\nu^{m}}{|\nu-s|^{m}}+|\theta_{2}||h^{*}(s)|(2\sum_{i=0}^{m-1}\binom{m+i-1}{i})\frac{\nu^{m+i}}{|2\nu-s|^{m+i}}+\frac{\nu^{m}}{|\nu-s|^{m}})\right]\to 0,
    \end{array}
\end{displaymath}
on $C$, i.e., as $r\to \infty$, and thus, $|q\tilde{J}(s;\theta_{2})|<1$, on $C$.

Denote by $d(s):=2\sum_{i=0}^{m-1}\binom{m+i-1}{i})(\frac{\nu}{2\nu-s})^{m+i}-(\frac{\nu}{\nu-s})^{m}$. Note that,
\begin{displaymath}
    d(0)=2\sum_{i=0}^{m-1}\binom{m+i-1}{i})(\frac{1}{2})^{m+i}-1=2(\frac{1}{2})^{m}\sum_{i=0}^{m-1}\binom{m+i-1}{i})(\frac{1}{2})^{i}-1=2(\frac{1}{2})^{m}2^{m-1}-1=0.
\end{displaymath}

For $s$ such that $Re(s)=0$,
\begin{displaymath}
    \begin{array}{rl}
        |q\tilde{J}(s;\theta_{2})|\leq &q\left[|\phi(s)|\frac{\nu^{m}}{|\nu-s|^{m}}+|\theta_{2}||h^{*}(s)|(2\sum_{i=0}^{m-1}\binom{m+i-1}{i})\frac{\nu^{m+i}}{|2\nu-s|^{m+i}}+\frac{\nu^{m}}{|\nu-s|^{m}})\right]\vspace{2mm}\\
        \leq& q\left[\frac{\nu^{m}}{|\nu-s|^{m}}+|\theta_{2}||d(s)|\right]\leq q(1+|\theta_{2}||d(0)|)=q<1.
    \end{array}
\end{displaymath}
Thus, in each case we proved that $|q\tilde{J}(s;\theta_{2})|<1$, or equivalently, 
\begin{equation*}
        \begin{array}{r}
            |q\left[\nu^{m}\phi(s)(2\nu-s)^{2m-1}+\theta\nu^{m}h^{*}(s)\left(2\sum_{i=0}^{m-1}\binom{m+i-1}{i}\nu^{i}(\nu-s)^{m}(2\nu-s)^{m-i-1}-(2\nu-s)^{2m-1}\right)\right]|\vspace{2mm}\\
<|(\nu-s)^{m}(2\nu-s)^{2m-1}|,
        \end{array}\label{bhjs}
    \end{equation*}
thus by Rouch\'e’s theorem \cite{tit}, it follows that \eqref{bhjx} has the same number of roots with $(\nu-s)^{m}(2\nu-s)^{2m-1}=0$ inside $C_r$. Since the latter
equation has exactly $3m-1$ positive roots inside $C_r$ (with $r\to\infty$), we deduce that \eqref{bhjx}, or equivalently $1-q\tilde{J}(s;\theta_{2})=0$ has exactly $3m-1$ roots, say $s_{1},\ldots,s_{3m-1}$ with positive
real parts.\end{proof}
\end{remark}

Iterating \eqref{fun} $N-1$ times yields,
\begin{displaymath}
    \rho(s)=\sum_{k=0}^{N-1}\prod_{j=0}^{k-1}J_{1}(\frac{s}{(1+a)^{j}}   ;\theta_{1},\theta_{2})H_{1}(\frac{s}{(1+a)^{k}};\theta_{1},\theta_{2})+\rho(\frac{s}{(1+a)^{N}})\prod_{j=0}^{N-1}J(\frac{s}{(1+a)^{j}};\theta_{1},\theta_{2}).
\end{displaymath}
Using similar arguments as above, we conclude that
\begin{equation}
 \rho(s)=\sum_{k=0}^{\infty}\prod_{j=0}^{k-1}J_{1}(\frac{s}{(1+a)^{j}}   ;\theta_{1},\theta_{2})H_{1}(\frac{s}{(1+a)^{k}};\theta_{1},\theta_{2}).\label{soll}
 \end{equation}
 
We still need to obtain $\rho^{(l)}(\mu)$, $l=0,1,\ldots,n-1$, $\rho^{(l)}(2\mu)$ $l=0,1,\ldots,2n-2$, and $\rho^{(d)}(\nu)$, $d=0,1,\ldots,m-1$, $\rho^{(d)}(2\nu)$, $d=0,1,\ldots,2m-2$. Differentiating $l$ times ($l=0,1,\ldots,n-1$) \eqref{soll} and substituting $s=\mu$ we obtain $n$ equations. Similarly, by differentiating $l$ times ($l=0,1,\ldots,2n-2$) \eqref{soll} and substituting $s=2\mu$ we have other $2n-1$. We still need other $3m-1$ equations. To this end, we invoke Rouch\'e's theorem \cite{tit} and Proposition \ref{propo}, where we showed that $1-p\tilde{J}(s;\theta_{2})=0$ has exactly $3m-1$ roots with positive real parts. Setting $s=s_{j}$, $j=1,\ldots,3m-1$ in \eqref{fun1} yields (since $\rho(s)$ is analytic in the right half $s$-plane, $\rho(s_{j})$, $j=1,\ldots,3m-1$ is finite)
\begin{equation}
    \frac{1-\phi(s_{j})}{s_{j}}+p(J(s_{j};\theta_{1})\rho(\frac{s_{j}}{1+a})+\widehat{H}(s_{j};\theta_{1}))+q\tilde{H}(s_{j};\theta_{2})=0,\,j=1,\ldots,3m-1.\label{nml}
\end{equation}
Definitely \eqref{nml} provide $3m-1$ equations, but they also introduce the additional unknowns $\rho(\frac{s_{j}}{1+a})$, $j=1,\ldots,3m-1$. However, this problem is overcome by substituting $s=\frac{s_{j}}{1+a}$ in \eqref{soll}, so that another equation is derived (for each $j=1,\ldots,3m-1$) where $\rho(\frac{s_{j}}{1+a})$ is given in terms of the unknowns $\rho^{(l)}(\mu)$, $\rho^{(l)}(2\mu)$, $\rho^{(d)}(\nu)$, $\rho^{(d)}(2\nu)$.
\subsection{An extension to the generalized FGM copula}
We now assume that the gain interarrivals and the gain sizes are dependent based on a generalized FGM copula, which belongs to a family of copulas introduced by \cite{rod}, and defined by
\begin{displaymath}
    C(u, v) = uv + \theta p(u)g(v),
\end{displaymath}
where $p(.)$ and $g(.)$ are non-zero real functions with support $[0, 1]$. For other extensions of the classical
FGM copula see e.g., \cite{drouet}. Our motivation on the extensions of FGM is to improve the range of dependence association (as measured by
either Kendall’s $\tau$ or Spearman’s $\rho$) between the components
of $(B,C)$. In this work we assume the case where $p(u):=u^{k}(1-u)^{b}$, $g(v):=v^{c}(1-v)^{d}$, with $k,b,c,d\geq 1$. The pdf associated to the GFGM is given by $c(u,v)=1+\theta p^{\prime}(u)g^{\prime}(v)$. Thus, the joint pdf of the random vectors $(B_{i},C_{i})$ is given by,
\begin{displaymath}
\begin{array}{rl}
     f_{B,C}(x,t)=&c(F_{B}(x),F_{C}(t))f_{B}(x)f_{C}(t)= f_{B}(x)f_{C}(t)+\theta p^{\prime}(F_{B}(x))g^{\prime}(F_{C}(t))f_{B}(x)f_{C}(t).
\end{array}
\end{displaymath}
Therefore,
\begin{equation}
    \begin{array}{rl}
         R(x)=&1-B(x)+\int_{t=0}^{x}\int_{y=0}^{\infty}R((1+a)(x-t)+y)f_{B,C}(t,y)dydt.
    \end{array}\label{eqz11}
\end{equation}
Assuming that $C_{i}$ are iid random variables that follow $exp(\mu)$ (the analysis is still applicable if we consider hyperexponential or Erlang distribution, but to keep the model as simple we can for the shake of readability we assume exponential distribution for the gain sizes), and applying Laplace transforms we come up with the following functional equation:
\begin{displaymath}
    \begin{array}{rl}
        \rho(s)= &\frac{1-\phi(s)}{s}+ \phi(s)\frac{\mu}{\mu(1+a)-s}(\rho(\frac{s}{1+a})-\rho(\mu))\vspace{2mm}\\
         &+\theta\int_{x=0}^{\infty}e^{-sx}\int_{t=0}^{x}\int_{y=0}^{\infty}R((1+a)(x-t)+y)p^{\prime}(F_{B}(t))g^{\prime}(1-e^{-\mu y})f_{B}(t)\mu e^{-\mu y}dydtdx,
    \end{array}
\end{displaymath}
where the functions $g$ and $h$ are defined above. In order to
rewrite the third term in the right-hand side of the above equation, let us define the r.v. $Z$ with pdf given by $g_{Z}(t)=f_{B}(t)-f_{B}(t)p^{\prime}(F_{B}(t))$, and the corresponding Laplace transform $g_{Z}^{*}(s)$. Let also define the
function $k_{C}(y)=g^{\prime}(1-e^{-\mu y})\mu e^{-\mu y}$. Thus,
\begin{equation}
    \begin{array}{rl}
        \rho(s)= &\frac{1-\phi(s)}{s}+ \phi(s)\frac{\mu}{\mu(1+a)-s}(\rho(\frac{s}{1+a})-\rho(\mu))\vspace{2mm}\\
         &+\theta\int_{x=0}^{\infty}e^{-sx}\int_{t=0}^{x}\int_{y=0}^{\infty}R((1+a)(x-t)+y)k_{C}(y)[f_{B}(t)-g_{Z}(t)]dydtdx.
    \end{array}\label{polp}
\end{equation}

By the definition of the function $g$, simple computations imply that
\begin{displaymath}
\begin{array}{rl}
    k_{C}(y)= &  g^{\prime}(1-e^{-\mu y})\mu e^{-\mu y}=\mu(1-e^{-\mu y})^{c-1}[(c+d)e^{-\mu(d+1)y}-de^{-\mu dy}]\vspace{2mm}\\
     =&\mu\sum_{i=0}^{c-1}\binom{c-1}{i}(-1)^{i} [(c+d)e^{-\mu(d+c-i)y}-de^{-\mu (c-1-i+d)y}].
\end{array}
\end{displaymath}
Then, the triple integral in the right-hand side of \eqref{polp}, can be rewritten as follows:
\begin{displaymath}
    \begin{array}{rl}
        I(s)= &\theta(\phi(s)-g_{Z}^{*}(s))\int_{t=x}^{\infty}e^{-s(x-t)}\int_{z=(1+a)(x-t)}^{\infty}R(z)k_{C}(z-(1+a)(x-t))dzdx\vspace{2mm}  \\
         =& \theta(\phi(s)-g_{Z}^{*}(s))\mu\sum_{i=0}^{c-1}\binom{c-1}{i}(-1)^{i}\left\{(c+d)\int_{w=0}^{\infty}e^{-w(s-\mu(1+a)(c+d-i))}\int_{z=(1+a)w}^{\infty}R(z)e^{-\mu(c+d-i)z}dzdw\right.\vspace{2mm}\\
        & \left.-d\int_{w=0}^{\infty}e^{-w(s-\mu(1+a)(c+d-i-1))}\int_{z=(1+a)w}^{\infty}R(z)e^{-\mu(c+d-i-1)z}dzdw\right\}\vspace{2mm}  \\
         =& \theta(\phi(s)-g_{Z}^{*}(s))\mu\sum_{i=0}^{c-1}\binom{c-1}{i}(-1)^{i}\left\{(c+d)\int_{z=0}^{\infty}R(z)e^{-\mu(c+d-i)z}\frac{1-e^{-\frac{z}{1+a}(s-\mu(1+a)(c+d-i))}}{s-\mu(1+a)(c+d-i)}dz\right.\vspace{2mm}\\
        & \left.-d\int_{z=0}^{\infty}R(z)e^{-\mu(c+d-i-1)z}\frac{1-e^{-\frac{z}{1+a}(s-\mu(1+a)(c+d-i-1))}}{s-\mu(1+a)(c+d-i-1)}dz\right\}\vspace{2mm}\\
        =&\theta(\phi(s)-g_{Z}^{*}(s))\mu\sum_{i=0}^{c-1}\binom{c-1}{i}(-1)^{i}\left\{\frac{c+d}{s-\mu(1+a)(c+d-i)}(\rho(\mu(s+d-i))-\rho(\frac{s}{1+a}))\right.\vspace{2mm}\\
        &\left.-\frac{d}{s-\mu(1+a)(c+d-i-1)}(\rho(\mu(s+d-i-1))-\rho(\frac{s}{1+a}))\right\}.
    \end{array}
\end{displaymath}
Hence, \eqref{polp} is now rewritten as
\begin{equation}
    \rho(s)=J(s;\theta)\rho(\frac{s}{1+a})+H(s;\theta),\label{polp1}
\end{equation}
where now,
\begin{equation}
    \begin{array}{rl}
      J(s;\theta):=   &\frac{\mu\phi(s)}{\mu(1+a)-s}+\theta(\phi(s)-g_{Z}^{*}(s))\mu\sum_{i=0}^{c-1}\binom{c-1}{i}(-1)^{i}\left[\frac{c+d}{\mu(1+a)(c+d-i)-s}-\frac{d}{\mu(1+a)(c+d-i-1)-s}\right],\vspace{2mm}  \\
       H(s;\theta):=  & \frac{1-\phi(s)}{s}+\widehat{H}(s;\theta)\\=& \frac{1-\phi(s)}{s}-\frac{\mu\phi(s)\rho(\mu)}{\mu(1+a)-s}\vspace{2mm}\\&+\theta(\phi(s)-g_{Z}^{*}(s))\mu\sum_{i=0}^{c-1}\binom{c-1}{i}(-1)^{i}\left\{\frac{(c+d)\rho(\mu(c+d-i))}{s-\mu(1+a)(c+d-i)}-\frac{d\rho(\mu(c+d-i-1))}{s-\mu(1+a)(c+d-i-1)}\right\}.
    \end{array}
\label{dol1}
\end{equation}
Note that \eqref{polp1} has exactly the same form with \eqref{eqo}, although the functions $J(s;\theta)$, $H(s;\theta)$ are different both because the marginals are now exponentially distributed (when deriving \eqref{eqo} we assumed they are Erlang distributed), and especially because we now used the generalized FGM copula (instead of the standard FGM copula).

Iterating $N-1$ times \eqref{eqo} results in
\begin{displaymath}
    \rho(s)=\sum_{k=0}^{N-1}\prod_{j=0}^{k-1}J(\frac{s}{(1+a)^{j}}   ;\theta)H(\frac{s}{(1+a)^{k}};\theta)+\rho(\frac{s}{(1+a)^{N}})\prod_{j=0}^{N-1}J(\frac{s}{(1+a)^{j}};\theta),
\end{displaymath}
with an empty product being equal to 1. Observe that for large $k$, $J(\frac{s}{(1+a)^{k}}   ;\theta)$ approaches
\begin{displaymath}
    \begin{array}{l}
         \frac{\phi(0)}{1+a}+\theta(\phi(0)-g_{Z}^{*}(0))\sum_{i=0}^{c-1}\binom{c-1}{i}(-1)^{i}\left[\frac{c+d}{(1+a)(c+d-i)}-\frac{d}{(1+a)(c+d-i-1)}\right]=\frac{\phi(0)}{1+a}=\frac{1}{1+a}<1,
    \end{array}
\end{displaymath}
and $H(\frac{s}{(1+a)^{k}};\theta)$ a constant. Thus,
 \begin{equation}
    \rho(s)=\sum_{k=0}^{\infty}\prod_{j=0}^{k-1}J(\frac{s}{(1+a)^{j}}   ;\theta)H(\frac{s}{(1+a)^{k}};\theta).\label{gahy1}    
    \end{equation}
    The values of $\rho(\mu)$, $\rho(\mu(c+d-i))$, $i=0,\ldots,c$, can be derived by substituting $s=\mu$, $s=\mu(c+d-i)$, $i=0,\ldots,c$, in \eqref{gahy1}, and solving a system of $c+2$ equations.
\begin{remark}
    It would be interesting to consider the case where with probability $p$ the $i$th jump has size $au+C_{i}$ (when $U(T_{i}^{-})=u$), and with probability $q:=1-p$, has size $D_{i}\sim\exp(\nu)$, where now the components of $(B_{i},C_{i})$ are dependent through the GFGM with parameters $(\theta_{1},k_{1},b_{1},c_{1},d_{1})$, and the components of $(B_{i},D_{i})$ are dependent through the GFGM with parameters $(\theta_{2},k_{2},b_{2},c_{2},d_{2})$, where $\theta_{1}\neq\theta_{2}, ,k_{1}\neq k_{2},b_{1}\neq b_{2},c_{1}\neq c_{2},d_{1}\neq d_{2}$. Then,
    \begin{equation}
    \begin{array}{rl}
        \rho(s)= &\frac{1-\phi(s)}{s}+ p\left\{\phi(s)\frac{\mu}{\mu(1+a)-s}(\rho(\frac{s}{1+a})-\rho(\mu))\right.\vspace{2mm}\\
         &\left.+\theta_{1}\int_{x=0}^{\infty}e^{-sx}\int_{t=0}^{x}\int_{y=0}^{\infty}R((1+a)(x-t)+y)k_{C}(y)[f_{B}(t)-g_{Z}(t)]dydtdx\right\}\vspace{2mm}\\
         &+q\left\{\phi(s)\frac{\nu}{\nu-s}(\rho(s)-\rho(\nu))\right.\vspace{2mm}\\
         &\left.+\theta_{2}\int_{x=0}^{\infty}e^{-sx}\int_{t=0}^{x}\int_{y=0}^{\infty}R(x-t+y)k_{D}(y)[f_{B}(t)-\tilde{g}_{Z}(t)]dydtdx\right\},
    \end{array}\label{polpX}
\end{equation}
where $\tilde{g}_{Z}(t)=f_{B}(t)-f_{B}(t)p_{*}^{\prime}(F_{B}(t))$, with $p_{*}(u):=u^{k_{2}}(1-u)^{b_{2}}$. Simple but tedious computations yield
\begin{equation}
    \rho(s)=J_{1}(s;\theta_{1},\theta_{2})\rho(\frac{s}{1+a})+H_{1}(s;\theta_{1},\theta_{2}),\label{funx}
\end{equation}
where,
\begin{displaymath}
    \begin{array}{rl}
        J_{1}(s;\theta_{1},\theta_{2}):= &\frac{pJ(s;\theta_{1})}{1-qG(s;\theta_{2})},\vspace{2mm}  \\
        G(s;\theta_{2}):= & \frac{\nu\phi(s)}{\nu-s}+\theta_{2}(\phi(s)-\tilde{g}_{Z}^{*}(s))\nu\sum_{i=0}^{c_{2}-1}\binom{c_{2}-1}{i}(-1)^{i}\left[\frac{c_{2}+d_{2}}{\nu(c_{2}+d_{2}-i)-s}-\frac{d}{\nu(c_{2}+d_{2}-i-1)-s}\right],\vspace{2mm}\\
        H_{1}(s;\theta_{1},\theta_{2}):=&\frac{\frac{1-\phi(s)}{s}+p\widehat{H}(s;\theta_{1})+q\tilde{H}(s;\theta_{2})}{1-qG(s;\theta_{2})},\vspace{2mm}\\
        \tilde{H}(s;\theta_{2}):=&\frac{\nu\phi(s)\rho(\nu)}{\nu-s}\vspace{2mm}\\&+\theta_{2}(\phi(s)-\tilde{g}_{Z}^{*}(s))\nu\sum_{i=0}^{c_{2}-1}\binom{c_{2}-1}{i}(-1)^{i}\left\{\frac{(c_{2}+d_{2})\rho(\nu(c_{2}+d_{2}-i))}{s-\nu(c_{2}+d_{2}-i)}-\frac{d\rho(\nu(c_{2}+d_{2}-i-1))}{s-\nu(c_{2}+d_{2}-i-1)}\right\},
    \end{array}
\end{displaymath}
and $J(s;\theta_{1})$, $\widehat{H}(s;\theta_{1})$ as given in \eqref{dol1}, where $\theta_{1}\equiv\theta$, and $c\equiv c_{1}$, $d\equiv d_{1}$, $k\equiv k_{1}$, $b\equiv b_{1}$.

By letting $\nu_{i}:=\nu(d_{2}+i-1)$, $i=1,\ldots,c_{2}+1$, $G(s;\theta_{2})$ is now rewritten as
\begin{displaymath}\begin{array}{l}
    G(s;\theta_{2}):=\frac{\nu\phi(s)\prod_{i=1}^{c_{2}+1}(\nu_{i}-s)+\theta_{2}(\phi(s)-\tilde{g}_{Z}^{*}(s))\nu(\nu-s)\sum_{j=0}^{c_{2}-1}\binom{c_{2}-1}{j}(-1)^{j}\Bigl(c(\nu_{c+1-j}-s)-\nu_{c+1}\Bigr)\prod_{l\neq,c-j,c+1-j}(\nu_{l}-s)}{(\nu-s)\prod_{i=1}^{c_{2}+1}(\nu_{i}-s)}.\end{array}
\end{displaymath}
Note that \eqref{funx} has the same form as the one in \eqref{polp1}, and its solution will have the form of \eqref{gahy1}, i.e., 
\begin{equation}
    \rho(s)=\sum_{k=0}^{\infty}\prod_{j=0}^{k-1}J_{1}(\frac{s}{(1+a)^{j}}   ;\theta_{1},\theta_{2})H_{1}(\frac{s}{(1+a)^{k}};\theta_{1},\theta_{2}).\label{gahy2}    
    \end{equation}
    However, we still have to obtain $c_{1}+c_{2}+4$ unknown terms, namely, $\rho(\mu)$, $\rho(\mu(d_{1}+i-1))$, $i=1,\ldots,c_{1}+1$, and $\rho(\nu)$, $\rho(\nu(d_{2}+j-1))$, $j=1,\ldots,c_{2}+1$. Clearly, $c_{1}+2$ equations can be derived by simply substituting $s=\mu$, and $s=\mu (d_{1}+i-1)$, $i=1,\ldots,c_{1}+1$ in \eqref{gahy2}. In order to construct other $c_{2}+2$, we need to show that the equation $1-qJ_{2}(s;\theta_{2})=0$ has exactly $c_{2}+2$ roots, say $s_{1},\ldots,s_{c+2}$, such that $Re(s_{k})>0$, $k=1,\ldots,c_{2}+2$. This task can be accomplished by using Rouch\'e's theorem \cite{tit} as we did in proposition \ref{propo}. However, by substituting $s=s_{k}$ in 
    \begin{equation}
    \rho(s)[1-qG(s;\theta_{2})]=\frac{1-\phi(s)}{s}+p(J(s;\theta_{1})\rho(\frac{s}{1+a})+\widehat{H}(s;\theta_{1}))+q\tilde{H}(s;\theta_{2}),\label{fun1x}
\end{equation}
another unknown is introduced, i.e., $\rho(\frac{s_{k}}{1+a})$. However, this unknown can be expressed in terms of $\rho(\mu)$, $\rho(\mu(d_{1}+i-1))$, $i=1,\ldots,c_{1}+1$, and $\rho(\nu)$, $\rho(\nu(d_{2}+j-1))$, $j=1,\ldots,c_{2}+1$, through \eqref{gahy2}.
\begin{proposition}\label{propo1}
    For $\theta_{2}\neq 0$, the equation $qG(s;\theta_{2})=1$ has exactly $c_{2}+2$ roots, say $s_{1},s_{2},\ldots,s_{c_{2}+2}$ in the right-half complex plane, i.e., $Re(s_{j})>0$, $j=1,2,\ldots,c_{2}+2$.
\end{proposition}
\begin{proof}
    The proof is based on Rouch\'e's theorem \cite{tit} on a contour $C_r$, consisting of the imaginary axis running from $-ir$ to $ir$ and a semi-circle with
radius $r$ running clockwise from $ir$ to $-ir$. Let $r\to\infty$ and
denote by $C$ the limiting contour. It suffices to show that the equation
\begin{equation}\begin{array}{l}
 (\nu-s)\prod_{i=1}^{c_{2}+1}(\nu_{i}-s) =q\nu\left\{\phi(s)\prod_{i=1}^{c_{2}+1}(\nu_{i}-s)\right.\vspace{2mm}\\
 \left.+\theta_{2}(\phi(s)-\tilde{g}_{Z}^{*}(s))(\nu-s)\sum_{j=0}^{c_{2}-1}\binom{c_{2}-1}{j}(-1)^{j}\Bigl(c(\nu_{c+1-j}-s)-\nu_{c+1}\Bigr)\prod_{l\neq,c-j,c+1-j}(\nu_{l}-s)\right\},
 \end{array}\label{huio}
\end{equation}
has exactly $c_{2}+2$ roots with positive real parts. Equivalently, we need to show that $|qG(s;\theta_{2})|<1$ on $C$. Note that $\frac{\nu}{\nu-s}$ and
\begin{displaymath}
   d(s):= \frac{\nu\sum_{j=0}^{c_{2}-1}\binom{c_{2}-1}{j}(-1)^{j}\Bigl(c(\nu_{c+1-j}-s)-\nu_{c+1}\Bigr)\prod_{l\neq,c-j,c+1-j}(\nu_{l}-s)}{\prod_{i=1}^{c_{2}+1}(\nu_{i}-s)},
\end{displaymath}
are ratios of polynomials with a strictly higher degree at the
denominator. Thus, $|qG(s;\theta_{2})|\to 0$ on $C$ (except from $Re(s)=0$). For $Re(s)=0$,
\begin{displaymath}
    \begin{array}{rl}
    |qG(s;\theta_{2})|\leq     &q\left(|\phi(s)\frac{\nu}{\nu-s}|+|\theta_{2}||\phi(s)-\tilde{g}_{Z}(s)||d(s)|\right)  \\
        \leq &q\left(|\frac{\nu}{\nu-s}|+|\theta_{2}||d(s)|\right)\leq q(1+\theta_{2}|d(0)|)=q<1, 
    \end{array}
\end{displaymath}
since simple computations imply that $d(0)=0$. Therefore, Rouch\'e's theorem implies that \eqref{huio}, or equivalently $qG(s;\theta_{2})=1$ has exactly $c_{2}+2$ roots, say $s_{j}$, $j=1,\ldots,c_{2}+2$ such that $Re(s_{j})>0$.

    \end{proof}
\end{remark}

\subsection{The dual risk model with proportional gains and a linear dependence among gain interarrival times and surplus level}
Following the notion of this section, we consider the case where gain interarrivals and gain sizes are dependent based on an FGM copula. On top of that, we further assume that when the surplus level is $x$ the next gain interarrival time is $cx$, $c\in (0,1)$. Therefore, we assume that the gain interarrival times are linearly dependent on the surplus level.  For convenience we assume that the gain interarrivals are exponentially distributed with rate $\lambda$ and gain sizes are exponentially distributed with rate $\mu$. Then,
\begin{equation}
    \begin{array}{rl}
         R(x)=&e^{-\lambda(1-c)x}+\int_{t=cx}^{x}\int_{y=0}^{\infty}R((1+a)(x-t)+y)f_{B,C}(t-cx,y)dydt,
    \end{array}\label{eqz11m}
\end{equation}
where
\begin{displaymath}
    f_{B,C}(t,y)=\lambda\mu e^{-\lambda t}\mu e^{-\mu y}+\theta(2\lambda e^{-2\lambda t}-\lambda e^{-\lambda t})(2\mu e^{-2\mu y}-\mu e^{-\mu y}),\,x,y\geq 0,\theta\in[-1,1].
\end{displaymath}
\begin{remark}
    Note that for $\theta=0$, i.e., the independent case, and $a=0$, our model reduces to the model in \cite[Section 3]{boxman}.
\end{remark}
Taking Laplace transforms, setting $x-t=w$ and $\bar{c}:=1-c$, we obtain after some algebra:
\begin{displaymath}
    \begin{array}{rl}
         \rho(s)=&\frac{1}{s+\lambda\bar{c}}+\lambda\mu\int_{x=0}^{\infty}e^{-sx}\int_{w=0}^{\bar{c}x}\int_{y=0}^{\infty}R(w(1+a)+y)\left[(1+\theta) e^{-\lambda (\bar{c}x-w)} e^{-\mu y} +4\theta  e^{-2\lambda (\bar{c}x-w)} e^{-2\mu y} \right.\vspace{2mm}\\
         &\left.-2\theta e^{-2\lambda (\bar{c}x-w)} e^{-\mu y}-2\theta  e^{-\lambda (\bar{c}x-w)} e^{-2\mu y}\right]dydwdx\vspace{2mm}\\
         =& \frac{1}{s+\lambda\bar{c}}+\lambda\mu[(1+\theta)I_{1}(s)+4\theta I_{2}(s)-2\theta(I_{3}(s)+I_{4}(s))],
    \end{array}
\end{displaymath}
where,
\begin{displaymath}
    \begin{array}{rl}
        I_{1}(s)= & \int_{x=0}^{\infty}e^{-sx}\int_{w=0}^{\bar{c}x}\int_{y=0}^{\infty}R(w(1+a)+y) e^{-\lambda (\bar{c}x-w)} e^{-\mu y}dydwdx\vspace{2mm} \\
         =& \int_{x=0}^{\infty}e^{-sx}\int_{w=0}^{\bar{c}x}\int_{z=w(1+a)}^{\infty}R(z) e^{-\lambda (\bar{c}x-w)}\mu e^{-\mu (z-w(1+a))}dydwdx\vspace{2mm}\\
         =&\int_{w=0}^{\infty}e^{w(\lambda+\mu(1+a))}\int_{x=w/\bar{c}}^{\infty}e^{-x(s+\lambda\bar{c})}\int_{z=w(1+a)}^{\infty}R(z)  e^{-\mu z}dzdxdw\vspace{2mm}\\
         =&\frac{1}{s+\lambda\bar{c}}\int_{w=0}^{\infty}e^{-w(\frac{s}{\bar{c}}-\mu(1+a))}\int_{z=w(1+a)}^{\infty}R(z)  e^{-\mu z}dzdw\vspace{2mm}\\
         =&\frac{\bar{c}}{(s+\lambda\bar{c})(s-\bar{c}\mu(1+a))}(\rho(\mu)-\rho(\frac{s}{\bar{c}(1+a)})).
    \end{array}
\end{displaymath}
Similarly,
\begin{displaymath}
    \begin{array}{rl}
        I_{2}(s)= & \int_{x=0}^{\infty}e^{-sx}\int_{w=0}^{\bar{c}x}\int_{y=0}^{\infty}R(w(1+a)+y) e^{-2\lambda (\bar{c}x-w)} e^{-2\mu y}dydwdx\vspace{2mm} \\
         =& \int_{x=0}^{\infty}e^{-sx}\int_{w=0}^{\bar{c}x}\int_{z=w(1+a)}^{\infty}R(z) e^{-2\lambda (\bar{c}x-w)}\mu e^{-2\mu (z-w(1+a))}dydwdx\vspace{2mm}\\
         =&\int_{w=0}^{\infty}e^{w(2\lambda+2\mu(1+a))}\int_{x=w/\bar{c}}^{\infty}e^{-x(s+2\lambda\bar{c})}\int_{z=w(1+a)}^{\infty}R(z)  e^{-2\mu z}dzdxdw\vspace{2mm}\\
         =&\frac{1}{s+2\lambda\bar{c}}\int_{w=0}^{\infty}e^{-w(\frac{s}{\bar{c}}-2\mu(1+a))}\int_{z=w(1+a)}^{\infty}R(z)  e^{-2\mu z}dzdw\vspace{2mm}\\
         =&\frac{\bar{c}}{(s+2\lambda\bar{c})(s-\bar{c}2\mu(1+a))}(\rho(2\mu)-\rho(\frac{s}{\bar{c}(1+a)})).\vspace{2mm}\\
         I_{3}(s)= & \int_{x=0}^{\infty}e^{-sx}\int_{w=0}^{\bar{c}x}\int_{y=0}^{\infty}R(w(1+a)+y) e^{-2\lambda (\bar{c}x-w)} e^{-\mu y}dydwdx\vspace{2mm}\\
         =&\frac{\bar{c}}{(s+2\lambda\bar{c})(s-\bar{c}\mu(1+a))}(\rho(\mu)-\rho(\frac{s}{\bar{c}(1+a)})).\vspace{2mm}\\
         I_{4}(s)= & \int_{x=0}^{\infty}e^{-sx}\int_{w=0}^{\bar{c}x}\int_{y=0}^{\infty}R(w(1+a)+y) e^{-\lambda (\bar{c}x-w)} e^{-2\mu y}dydwdx\vspace{2mm}\\
         =&\frac{\bar{c}}{(s+\lambda\bar{c})(s-\bar{c}2\mu(1+a))}(\rho(2\mu)-\rho(\frac{s}{\bar{c}(1+a)})).
    \end{array}
\end{displaymath}
To conclude, the Laplace transform of the ruin probability starting at surplus level $x$ satisfies the following functional equation:
\begin{equation}
    \rho(s)=J(s;\theta)\rho(\zeta(s))+H(s;\theta),\label{polo}
\end{equation}
where now, $\zeta(s):=\frac{s}{\bar{c}(1+a)}$, and,
\begin{displaymath}
    \begin{array}{rl}
         J(s;\theta)=&\frac{\lambda\mu\bar{c}[(s+2\lambda\bar{c})(2\mu\bar{c}(1+a)-s)-\theta s^{2}]}{(s+2\lambda\bar{c})(2\mu\bar{c}(1+a)-s)(s+\lambda\bar{c})(\mu\bar{c}(1+a)-s)},\vspace{2mm}  \\
         H(s;\theta)=&\frac{1}{s+\lambda\bar{c}}\left(1-\lambda\mu\bar{c}\left[\frac{s(1-\theta)+2\lambda\bar{c}}{(s+2\lambda\bar{c})(\mu\bar{c}(1+a)-s)}\rho(\mu)+\frac{2\theta s}{(s+2\lambda\bar{c})(2\mu\bar{c}(1+a)-s)}\rho(2\mu)\right]\right).
    \end{array}
\end{displaymath}
Note that \eqref{polo} has the same form as the one in \cite[eq. (2.1)]{boxruin}. Iteration of \eqref{polo} results in the following theorem.
\begin{theorem}
    The Laplace transform of the ruin probability $\rho(s)$ is given by
    \begin{equation}
        \rho(s)=\sum_{k=0}^{\infty}\prod_{j=0}^{k-1}J(\zeta^{(j)}(s);\theta)H(\zeta^{(k)}(s);\theta),\label{cvax}
    \end{equation}
    where $\zeta^{(j)}(s):=\zeta(\zeta^{(j-1)}(s))=\frac{s}{(\bar{c}(1+a))^{j}}$, with $\zeta^{(0)}(s)=s$. The unknowns $\rho(\mu)$, $\rho(2\mu)$ (that appear with a prefactor in all $H(\zeta^{(k)}(s);\theta)$) are derived by the system of equations that is constructed by substituting $s=\mu$, $s=2\mu$ in \eqref{cvax}.
\end{theorem}

The next step is to focus on the LST of the ruin time $\tau_{x}$, denoting by $T(s,x)$:
\begin{displaymath}\begin{array}{rl}
     T(s,x)=&e^{-sx}e^{-\lambda(1-c)x}+\int_{t=cx}^{x}e^{-st}\int_{y=0}^{\infty}T(s,(1+a)(x-t)+y)f_{B,C}(t-cx,y)dydt\vspace{2mm}\\
     =& e^{-x(s+\lambda\bar{c})}+\int_{w=0}^{\bar{c}x}e^{-s(x-w)}\int_{y=0}^{\infty}T(s,(1+a)w+y)f_{B,C}(\bar{c}x-w,y)dydw.
\end{array}
\end{displaymath}
Let $\tau(s,\beta)$ the Laplace transform of the ruin time LST $T(s,x)$. Then,
\begin{displaymath}
    \begin{array}{rl}
        \tau(s,\beta)= &\frac{1}{s+\lambda\bar{c}+\beta}+\int_{x=0}^{\infty}e^{-s\beta}\int_{w=0}^{\bar{c}x}e^{-s(x-w)}\int_{z=(1+a)w}^{\infty}R(z)f_{B,C}(\bar{c}x-w,z-w(1+a))dzdwdx.  \\
         & 
    \end{array}
\end{displaymath}
By repeating similar steps as above, we obtain after some algebra the following functional equation:
\begin{equation}
    \tau(s,\beta)=\tilde{J}(s,\beta;\theta)\tau(s,\psi(s,\beta))+\tilde{H}(s,\beta;\theta),\label{polo1}
\end{equation}
where now, $\psi(s,\beta):=\frac{r(s,\beta)}{1+a}=\frac{\beta+cs}{\bar{c}(1+a)}$, and,
\begin{displaymath}
    \begin{array}{rl}
         \tilde{J}(s,\beta;\theta)=&\frac{\lambda\mu[(s+2\lambda\bar{c}+\beta)(2\mu(1+a)-r(s,\beta))-\theta (s+\beta)r(s,\beta)]}{(s+2\lambda\bar{c}+\beta)(2\mu(1+a)-r(s,\beta))(s+\lambda\bar{c}+\beta)(\mu(1+a)-r(s,\beta))},\vspace{2mm}  \\
         \tilde{H}(s,\beta;\theta)=&\frac{1}{s+\lambda\bar{c}+\beta}\left(1-\lambda\mu\left[\frac{(s+\beta)(1-\theta)+2\lambda\bar{c}}{(s+2\lambda\bar{c}+\beta)(\mu(1+a)-r(s,\beta))}\tau(s,\mu)+\frac{2\theta (s+\beta)}{(s+2\lambda\bar{c}+\beta)(2\mu(1+a)-r(s,\beta))}\tau(s,2\mu)\right]\right).
    \end{array}
\end{displaymath}
Writing $\psi^{(j)}(s,\beta):=\psi(\psi^{(j-1)}(s,\beta))$ with $\psi^{(0)}(s,\beta):=\beta$, we have,
\begin{displaymath}
    \psi^{(j)}(s,\beta)=\frac{\beta}{(\bar{c}(1+a))^{j}}+\sum_{i=1}^{j}\frac{cs}{(\bar{c}(1+a))^{i}},\,j=1,2,\ldots.
\end{displaymath}
Iterating \eqref{polo1} we obtain the following result:
\begin{theorem}
    The Laplace transform of the ruin time LST $T(s,x)$ is given by
    \begin{equation}
        \tau(s,\beta)=\sum_{k=0}^{\infty}\prod_{j=0}^{k-1}\tilde{J}(\psi^{(j)}(s,\beta);\theta)\tilde{H}(\psi^{(k)}(s,\beta);\theta).\label{cqvax}
    \end{equation}
     The unknowns $\tau(s,\mu)$, $\tau(s,2\mu)$ (that appear with a prefactor in all $\tilde{H}(\psi^{(k)}(s,\beta;\theta)$) are derived by the system of equations that is constructed by substituting $s=\mu$, $s=2\mu$ in \eqref{cqvax}.
\end{theorem}
Note that the sums of products in \eqref{cvax}, \eqref{cqvax} converge since both $J(.,.)$, $H(.,.)$ (resp. $\tilde{J}(.,.)$, $\tilde{H}(.,.)$) decrease geometrically fast.
\begin{remark}
    Note that for $c=0$, $a=0$, we cope with the standard (no proportional gains) dual risk model with $\exp(\mu)$ initial capital, and for which there is a dependence among gain interarrivals and gain sizes based on the FGM copula. In queueing terms, we are dealing with an M/M/1 queue and studying the busy period starting from an empty system with an $\exp(\mu)$ upward jump, in which the interarrival times and work that enters the system are dependent based on the FGM copula. Note that $\mu\tau(s,\mu)=\int_{0}^{\infty}T(s,x)\mu e^{-\mu x}dx$ should be equal to the to the LST of an M/M/1 queue with initial service time to be $\exp(\mu)$, and for which we consider the dependence structure mentioned above. In particular, for $c=a=0$, \eqref{polo1} is rewritten as:
    \begin{equation}
    \tau(s,\beta)(1-\widehat{J}(s,\beta;\theta))=\widehat{H}(s,\beta;\theta),\label{polo11}
\end{equation}
where
\begin{displaymath}
    \begin{array}{rl}
       \widehat{J}(s,\beta;\theta)=  &\lambda\mu\frac{(s+2\lambda+\beta)(2\mu-\beta)-\theta(s+\beta)\beta}{(s+\lambda+\beta)(s+2\lambda+\beta)(\mu-\beta)(2\mu-\beta)}, \vspace{2mm} \\
       \widehat{H}(s,\beta;\theta)=  &\frac{1}{s+\lambda+\beta}\left(1-\lambda\mu\left[\frac{(s+\beta)(1-\theta)+2\lambda}{(s+2\lambda+\beta)(\mu-\beta)}\tau(s,\mu)+\frac{2\theta (s+\beta)}{(s+2\lambda+\beta)(2\mu-\beta)}\tau(s,2\mu)\right]\right).
    \end{array}
\end{displaymath}
It is readily seen that for $\theta=0$, \eqref{polo11} reduces to the one in \cite[eq. (29)]{boxman}. Moreover, taking $\beta=\mu$, and $\beta=2\mu$, \eqref{polo11} gives an identity. We now investigate the roots of the equation $1-\widehat{J}(s,\beta;\theta)=0$.
\begin{proposition}
    For $Re(s)>0$, $\theta\neq 0$, the equation $1-\widehat{J}(s,\beta;\theta)=0$ has exactly two roots, say $u_{1}(s),u_{2}(s)$ with $Re(u_{i}(s))>0$, $i=1,2$.
\end{proposition}
\begin{proof}
    Following Rouch\'e's theorem \cite{tit}, we consider a contour $C_r$, consisting of the imaginary axis running from $-ir$ to $ir$ and a semi-circle with radius $r$ running clockwise from $ir$ to $-ir$. Let $r\to\infty$ and denote by $C$ the limiting contour. Note that $1-\widehat{J}(s,\beta;\theta)=0$ is rewritten as 
   \begin{equation}
      \lambda\mu[(s+2\lambda+\beta)(2\mu-\beta)-\theta(s+\beta)\beta]=(s+\lambda+\beta)(s+2\lambda+\beta)(\mu-\beta)(2\mu-\beta).
       \label{gt}
    \end{equation}
    So it suffices to show that \eqref{gt} has exactly two roots in  the  right-half  complex  plane. 
    It is readily seen that $|\widehat{J}(s,\beta;\theta)|\to 0$ on $C$ (excluding $Re(\beta)=0$), since it is the sum of ratios of polynomials with a strictly higher degree at the denominator. For $Re(\beta)=0$,
    \begin{displaymath}
        \begin{array}{rl}
            |\widehat{J}(s,\beta;\theta)|= &\lambda\mu|\frac{(s+2\lambda+\beta)(2\mu-\beta)-\theta(s+\beta)\beta}{(s+\lambda+\beta)(s+2\lambda+\beta)(\mu-\beta)(2\mu-\beta)}|\leq \lambda\mu|\frac{1}{(s+\lambda+\beta)(\mu-\beta)}|+\lambda\mu|\theta|\frac{(s+\beta)\beta}{(s+\lambda+\beta)(s+2\lambda+\beta)(\mu-\beta)(2\mu-\beta)}|\\&\leq |\frac{s}{s+\lambda}|<1.
        \end{array}
    \end{displaymath}
    Thus in any case $|\widehat{J}(s,\beta;\theta)|<1$ or equivalently
    \begin{displaymath}
        \lambda\mu|(s+2\lambda+\beta)(2\mu-\beta)-\theta(s+\beta)\beta|<|(s+\lambda+\beta)(s+2\lambda+\beta)(\mu-\beta)(2\mu-\beta)|,
    \end{displaymath}
    and thus by Rouch\'e's theorem, it follows that \eqref{gt} has the same number of roots as  $(s+\lambda+\beta)(s+2\lambda+\beta)(\mu-\beta)(2\mu-\beta)=0$  inside $C_r$. Since  the  latter equation  has  exactly  two  positive  roots  inside $C_r$,  we  deduce  that \eqref{gt}, or equivalently $1-\widehat{J}(s,\beta;\theta)=0$ has exactly two roots, say $u_{1}(s),u_{2}(s)$ with positive real parts. Finally, we complete the proof by letting $r\to\infty$.
\end{proof}

Substituting in \eqref{polo11} $\beta=u_{1}(s)$, and $\beta=u_{2}(s)$ we construct a system of two equations with unknowns $\tau(s,\mu)$, $\tau(s,2\mu)$. By solving this system we are able to compute $\mu\tau(s,\mu)$ in terms of the $u_{1}(s),u_{2}(s)$. Note that by setting $\theta=0$ (i.e., the standard M/M/1 queue) $\mu\tau(s,\mu)$ coincides with the expression in \cite[eq. (29)]{boxman}.    
\end{remark}

\section{A dual risk model with upward and downward jumps and randomly proportional gains}\label{random}
We consider a dual risk model with constant expense rate. We focused on the case of two-sided jumps, i.e., the
upward and downward jumps can be interpreted as company random gains and random losses respectively. Thus, our model is suitable for
insurance companies with business in both property and casualty insurance and life
annuities. More precisely, we assume that with probability $p$ gains $C_i$, and with probability $q$ losses $-D_{i}$ ($i = 1, 2,\ldots $) that arrive according to a renewal process with general
interarrival times. Let, $C_{i}$s and $D_{i}$s are exponentially distributed (we also consider the case where they follow Erlang distributions). 

On top of that, we add the randomly proportional gain feature. More precisely, if the surplus process just before the $i$th arrival is at level $u$, then, the capital jumps to:
\begin{equation}
    \left\{\begin{array}{ll}
         (1+a_{l})u+C_{i},&\text{with probability }p\times k_{l},\,l=1,\ldots,K,  \\
         \left[(1+\beta_{h})u-D_{i}\right]^{+},& \text{with probability }q\times m_{h},\,h=1,\ldots,L,
    \end{array}\right.
\end{equation}
where $p+q=1$, and $\sum_{l=1}^{L}=1$, $\sum_{h=1}^{M}m_{h}=1$. Thus, with probability $p$ (resp. $q$) we have an upward (resp. downward) jump of size $C_{i}$ (resp. $-D_{i}$), with an additional inflow proportional to $u$, and equal to $a_{l}u$ (resp. $\beta_{h}u$) with probability $k_{l}$, $l=1,\ldots,K$ (resp. $m_{h}$, $h=1,\ldots,L$). So the type of the jump affects also the values of proportionality coefficient. Note that for $p=1$, $k_{1}=1$, our model reduces to the one in \cite{boxruin}. Then, having in mind that $[(1+\beta_{h})u-D_{i}]^{+}=\left\{\begin{array}{ll}
     (1+\beta_{h})u-D_{i},&(1+\beta_{h})u\geq D_{i},  \\
     0,&(1+\beta_{h})u<D_{i}, 
\end{array}\right.$ we have,
\begin{equation}
    \begin{array}{rl}
         R(x)=&1-B(x)+p\sum_{l=1}^{K}k_{l}\int_{t=0}^{x}\int_{y=0}^{\infty}R((1+a_{l})(x-t)+y)\mu e^{-\mu y}dydB(t)\vspace{2mm}\\
         &+q\sum_{h=1}^{K}m_{h}\int_{t=0}^{x}\left\{\int_{y=0}^{(1+\beta_{h})(x-t)}R((1+\beta_{h})(x-t)-y)\nu e^{-\nu y}dy+\int_{y=(1+\beta_{h})(x-t)}^{\infty}\nu e^{-\nu y}dy\right\}dB(t)
    \end{array}\label{eq1c}
\end{equation}
Applying Laplace transforms, we come up with the following functional equation,
\begin{equation}
    \rho(s)=\frac{1-\phi(s)}{s}+\sum_{l=1}^{K}p_{l}I_{l}(s)+\sum_{h=1}^{L}q_{h}\widehat{I}_{h}(s),\label{bnm}
\end{equation}
where $p_{l}=p\times k_{l}$, $l=1,\ldots,K$, and $q_{h}=q\times m_{h}$, $h=1,\ldots,L$. After lengthy, but straightforward computations,
\begin{equation}
    \begin{array}{rl}
        I_{l}(s)= &\frac{\phi(s)\mu}{s-\mu(1+a_{l})}(\rho(\mu)-\rho(\frac{s}{1+a_{l}})),\,l=1,\ldots,K,\vspace{2mm}  \\
         \widehat{I}_{h}(s)=&\frac{\phi(s)\nu}{s+\nu(1+\beta_{h})}(\rho(\frac{s}{1+\beta_{h}})+\frac{1}{\nu}),\,h=1,\ldots,L. 
    \end{array}\label{bvz}
\end{equation}
Substituting \eqref{bvz} in \eqref{bnm} we come up with the following functional equation:
\begin{equation}
    \rho(s)=\phi(s)\sum_{l=1}^{K+L}g_{l}f_{l}(\sigma_{l}(s))\rho(\zeta_{l}(s))+L(s),\label{bhy}
\end{equation}
where,
\begin{displaymath}
    g_{l}=\left\{\begin{array}{ll}
        \frac{p_{l}}{1+a_{l}}, &\,l=1,\ldots,K,  \\
       \frac{q_{l-K}}{1+\beta_{l-K}}, &\,l=K+1,\ldots,K+L,
    \end{array}\right.
\end{displaymath}
\begin{displaymath}
    f_{l}(s)=\left\{\begin{array}{ll}
         \frac{\mu}{\mu+s},&l=1,\ldots,K,  \\
         \frac{\nu}{\nu+s},&l=K+1,\ldots,K+L, 
    \end{array}\right.
\end{displaymath}
\begin{displaymath}
    \zeta_{l}(s)=\left\{\begin{array}{ll}
         a_{l}(s)=\frac{s}{1+a_{l}},&l=1,\ldots,K,  \\
         \widehat{a}_{l-K}(s)=\frac{s}{1+\beta_{l-K}},&l=K+1,\ldots,K+L, 
    \end{array}\right.
\end{displaymath}
\begin{displaymath}
    \sigma_{l}(s)=\left\{\begin{array}{ll}
         -a_{l}(s),&l=1,\ldots,K,  \\
         \widehat{a}_{l-K}(s),&l=K+1,\ldots,K+L, 
    \end{array}\right.
\end{displaymath}
and
\begin{displaymath}
    L(s)=\frac{1-\phi(s)}{s}+\phi(s)\sum_{l=1}^{K+L}\widehat{g}_{l}f_{l}(\sigma_{l}(s)),
\end{displaymath}
with $\widehat{g}_{l}=\left\{\begin{array}{ll}
         -g_{l}\rho(\mu),&l=1,\ldots,K,  \\
         \frac{g_{l}}{\nu},&l=K+1,\ldots,K+L.
    \end{array}\right.$

    Our aim now is to solve \eqref{bhy}. Note that the form of \eqref{bhy} is similar to \cite[eq. (2)]{adan}, and thus, a similar approach can be employed to solve it. After $N-1$ iterations, 
    \begin{displaymath}
    \begin{array}{rl}
        \rho(s)=& \sum_{k=0}^{N-1}\sum_{i_{1}+\ldots+i_{K+L}=k}g_{1}^{i_{1}}\ldots g_{K+L}^{i_{K+L}}G_{i_{1},\ldots,i_{K+L}}(s)L(\zeta_{i_{1},\ldots,i_{K+L}}(s))\vspace{2mm}   \\
         &+ \sum_{i_{1}+\ldots+i_{K+L}=k}g_{1}^{i_{1}}\ldots g_{K+L}^{i_{K+L}}G_{i_{1},\ldots,i_{K+L}}(s)\rho(\zeta_{i_{1},\ldots,i_{K+L}}(s)),
    \end{array}
    \end{displaymath}
    where $\zeta_{i_{1},\ldots,i_{K+L}}(s):=\zeta_{1}^{i_{1}}(\zeta_{2}^{i_{2}}(\ldots(\zeta_{K+L}^{i_{K+L}}(s))\ldots))$, and $\zeta_{k}^{m}(s)$ is the $m$th iterate of $\zeta_{k}(s)$, and the functions $G_{i_{1},\ldots,i_{K+L}}(s)$ are recursively obtained as follows:
    \begin{displaymath}
        G_{i_{1},\ldots,i_{K+L}}(s)=\sum_{k=1}^{K+L}G_{i_{1},\ldots,i_{k}-1,\ldots,i_{K+L}}(s)L_{0,\ldots,1,\ldots,0}(\zeta_{i_{1},\ldots,i_{k}-1,\ldots,i_{K+L}}(s)),
    \end{displaymath}
    with $G_{0,\ldots,0}(s)=1$, $G_{0,\ldots,1,\ldots,0}(s)=L_{0,\ldots,1,\ldots,0}(s)=\phi(s)f_{l}(\sigma_{l}(s))$, where 1 is in the $l$th position and $l=1,\ldots,K+L$, and $G_{i_{1},\ldots,i_{K+L}}(s)=0$ if one of the indices equals $-1$. Then, letting $N\to\infty$,
    \begin{equation}\begin{array}{rl}
    \rho(s)=&\sum_{k=0}^{\infty}\sum_{i_{1}+\ldots+i_{K+L}=k}g_{1}^{i_{1}}\ldots g_{K+L}^{i_{K+L}}G_{i_{1},\ldots,i_{K+L}}(s)L(\zeta_{i_{1},\ldots,i_{K+L}}(s))\vspace{2mm}\\&+ \lim_{N\to\infty}\sum_{i_{1}+\ldots+i_{K+L}=N}g_{1}^{i_{1}}\ldots g_{K+L}^{i_{K+L}}G_{i_{1},\ldots,i_{K+L}}(s), \end{array}\label{sollla}
\end{equation}
since $\rho(\zeta_{i_{1},\ldots,i_{K+L}}(s))\to\rho(0)=1$ since $\zeta_{l}(s)$, $l=1,\ldots,K+L$ are commutative contraction mappings on the closed positive half plane. It is readily seen that $G_{i_{1},\ldots,i_{K+L}}(s)$ can written as a finite sum of products. As the number of iterations increases, each of these products vanish. Moreover, as the number of iterations increases $L(\zeta_{i_{1},\ldots,i_{K+L}}(s))$ approaches some constant. Thus,
 \begin{equation}\begin{array}{rl}
    \rho(s)=&\sum_{k=0}^{\infty}\sum_{i_{1}+\ldots+i_{K+L}=k}g_{1}^{i_{1}}\ldots g_{K+L}^{i_{K+L}}G_{i_{1},\ldots,i_{K+L}}(s)L(\zeta_{i_{1},\ldots,i_{K+L}}(s)).\end{array}\label{solllta}
\end{equation}
We still need to derive $\rho(\mu)$. This can be achieved by substituting $s=\mu$ in \eqref{solllta}.
    \begin{remark}
        Consider the simpler case where $K=1=L$, and $a_{1}=\beta_{1}\equiv a$, so that \begin{equation}
    \left\{\begin{array}{ll}
         (1+a)u+C_{i},&\text{with probability }p,  \\
         \left[(1+a)u-D_{i}\right]^{+},& \text{with probability }q.
    \end{array}\right.\label{cvg}
\end{equation}
Then, \eqref{bhy} reduces to
\begin{equation}
    \rho(s)=\phi(s)[p\frac{\mu}{\mu(1+a)-s}+q\frac{\nu}{\nu(1+a)+s}]\rho(\frac{s}{1+a})+\frac{1-\phi(s)}{s}+\phi(s)[\frac{q}{\nu(1+a)+s}-\frac{p\mu}{\mu(1+a)-s}\rho(\mu)].\label{ghy}    
\end{equation}
Note that for $q=0$ (so that $p=1$) we recover the model in \cite{boxruin}. Equation \eqref{ghy} has the same form as in \cite[eq. (2.5)]{boxruin}, where now
\begin{displaymath}
    \begin{array}{rl}
        J(s)= & p\frac{\mu}{\mu(1+a)-s}+q\frac{\nu}{\nu(1+a)+s}, \\
         H(s)=& \frac{1-\phi(s)}{s}+\phi(s)[\frac{q}{\nu(1+a)+s}-\frac{p\mu}{\mu(1+a)-s}\rho(\mu)],
    \end{array}
\end{displaymath}
and its solution is as given in \cite[eq. (2.7)]{boxruin}, with $J(.)$, $H(.)$, as given above. 

In case where $a_{1}\equiv a\neq \beta_{1}\equiv \beta$, i.e., when \eqref{cvg} is given by,
\begin{equation}
    \left\{\begin{array}{ll}
         (1+a)u+C_{i},&\text{with probability }p,  \\
         \left[(1+\beta)u-D_{i}\right]^{+},& \text{with probability }q,
    \end{array}\right.\label{cvg1}
\end{equation}
and we now have to solve
\begin{equation}
    \rho(s)=\phi(s)J_{0}(s)\rho(\frac{s}{1+a})+\phi(s)J_{1}(s)\rho(\frac{s}{1+\beta})+H_{1}(s),\label{cxz}
\end{equation}
where,
\begin{displaymath}
    \begin{array}{rl}
        J_{0}(s)= & p\frac{\mu}{\mu(1+a)-s}\\
        J_{1}(s)=&q\frac{\nu}{\nu(1+\beta)+s}, \\
         H_{1}(s)=& \frac{1-\phi(s)}{s}+\phi(s)[\frac{q}{\nu(1+\beta)+s}-\frac{p\mu}{\mu(1+a)-s}\rho(\mu)].
    \end{array}
\end{displaymath}
The form of \eqref{cxz} is the same as the one in \eqref{eq3}, and thus, its solution can be derived similarly. 

Clearly, the dual risk model with proportional gains and two-sided jumps (i.e., either an upward or a downward jump) can be analysed similarly, when we consider the FGM copula to describe the dependence among the gain interarrival time and the corresponding size. Indeed, consider the simpler case where $C_{i}\sim \exp(\mu)$, $D_{i}\sim\exp(\nu)$ (the analysis is still applicable when we consider hyperexponential, or Erlang distributions), but there is a dependence with the ``gain" interarrival time based on the FGM copula, i.e., the joint density function of the random vectors $(B_{i},C_{i})$, $(B_{i},D_{i})$ are as follows: 
\begin{displaymath}
    \begin{array}{rl}
       f_{B,C}(t,y) = &f_{B}(t)\mu e^{-\mu y}+\theta_{1}h(t)(2\mu e^{-2\mu y}-\mu e^{-\mu y}),\,\theta_{1}\in[-1,1],  \\
        f_{B,D}(t,y) = &f_{B}(t)\nu e^{-\nu y}+\theta_{2}h(t)(2\nu e^{-2\nu y}-\nu e^{-\nu y}),\,\theta_{2}\in[-1,1], 
    \end{array}
\end{displaymath}
where $h(t)=f_{B}(t)(1-2F_{B}(t))$; see also Section \ref{copu}. Note that asking $\theta_{1}=0$ and/or $\theta_{2}=0$, we come up with the independent case.

Then, we have
\begin{equation}
    \begin{array}{rl}
         R(x)=&1-B(x)+p\int_{t=0}^{x}\int_{y=0}^{\infty}R((1+a)(x-t)+y)f_{B,C}(t,y)dydt\vspace{2mm}\\
         &+q\int_{t=0}^{x}\int_{y=0}^{(1+\beta)(x-t)}R((1+\beta)(x-t)-y) f_{B,D}(t,y)dydt+q\int_{y=(1+\beta)(x-t)}^{\infty}f_{B,C}(t,y)dydt,
    \end{array}\label{eq1d}
\end{equation}
Taking Laplace transforms, we come up with the functional equation
\begin{equation}
    \rho(s)=pJ_{1}(s;\theta_{1})\rho(\frac{s}{1+a})+qJ_{2}(s;\theta_{2})\rho(\frac{s}{1+\beta})+H(s;\theta_{1},\theta_{2}),\label{eqqq}
\end{equation}
where now,
\begin{displaymath}
    \begin{array}{rl}
       J_{1}(s;\theta_{1}):=  &\mu\frac{\phi(s)-\theta_{1}h^{*}(s)}{\mu(1+a)-s}+\frac{2\mu\theta_{1}h^{*}(s)}{2\mu(1+a)-s},  \vspace{2mm}\\
         J_{2}(s;\theta_{2}):=  &\nu\frac{\phi(s)-\theta_{2}h^{*}(s)}{\nu(1+\beta)+s}+\frac{2\nu\theta_{2}h^{*}(s)}{2\nu(1+\beta)+s}, \vspace{2mm} \\ 
         H(s;\theta_{1},\theta_{2}):=&\frac{1-\phi(s)}{s}+q[\frac{\phi(s)-\theta_{2}h^{*}(s)}{\nu(1+\beta)+s}+\frac{\theta_{2}h^{*}(s)}{2\nu(1+\beta)+s}]-p[\mu\rho(\mu)\frac{\phi(s)-\theta_{1}h^{*}(s)}{\mu(1+a)-s}+\rho(2\mu)\frac{2\mu\theta_{1}h^{*}(s)}{2\mu(1+a)-s}].
    \end{array}
\end{displaymath}
Note that \eqref{eq1d} has the same form as the one in \eqref{eq3}, so we can apply the same method to solve it. After $n-1$ ($n\geq 1$) iterations we have,
\begin{displaymath}
    \rho(s)=\sum_{k=0}^{n}\sum_{i_{1}+i_{2}=k}p^{i_{1}}q^{i_{2}}L_{i_{1},i_{2}}(s)H(f_{i_{1},i_{2}}(s);\theta_{1},\theta_{2})+\sum_{i_{1}+i_{2}=n}p^{i_{1}}q^{i_{2}}L_{i_{1},i_{2}}(s)\rho(f_{i_{1},i_{2}}(s)),
\end{displaymath}
where $L_{i_{1},i_{2}}(s)$ are recursively derived as in \eqref{assi}, and $f_{i_{1},i_{2}}(s)=a_{i_{1}}(b_{i_{2}}(s))=b_{i_{2}}(a_{i_{1}}(s))=\frac{s}{(1+a)^{i_{1}}(1+\beta)^{i_{2}}}$ with $a_{i}(s)=\frac{s}{(1+a)^{i}}$, $b_{j}(s)=\frac{s}{(1+\beta)^{j}}$. Note that $a_{i}(s)$, $b_{j}(s)$ are commutative contraction mappings on the right half plane, and the methodology presented in \cite[Section 2]{adan} applies. In particular,
\begin{equation}
    \rho(s)=\sum_{k=0}^{\infty}\sum_{i_{1}+i_{2}=k}p^{i_{1}}q^{i_{2}}L_{i_{1},i_{2}}(s)H(f_{i_{1},i_{2}}(s);\theta_{1},\theta_{2})+\lim_{n\to\infty}\sum_{i_{1}+i_{2}=n}p^{i_{1}}q^{i_{2}}L_{i_{1},i_{2}}(s).\label{solll}
\end{equation}
We still need to obtain $\rho(\mu)$, $\rho(2\mu)$. This can be done by solving a system of two equations that are constructed by setting $s=\mu$, and $s=2\mu$ in \eqref{solll}.
    \end{remark}
    \section{The dual risk model with uniformly proportional gains}\label{rv}
    Consider now the case where the proportional parameter is a random variable, i.e., the surplus process $U(t)$ with $U(0)=x>0$ evolves as
    \begin{displaymath}
        U(t)=x-t+\sum_{i=1}^{N(t)}(C_{i}+V_{i}U(S_{i}^{-})),\,t\geq 0,
    \end{displaymath}
    where $V_{1},V_{2},\ldots$ are i.i.d. uniformly distributed random variables on $[a,b]$, $0<a<b$. Moreover, $N(t)$ denotes the number of gains that arrive in $(0,t]$, with $S_{i+1}-S_{i}$ are i.i.d. gain interarrival times having c.d.f. $B(.)$, density $b(.)$, and LST $\phi(.)$. By assuming the $C_{i}$s are i.i.d. exponentially distributed random variables with rate $\mu$, the ruin probability $R(x)$ when starting in $x$ satisfies the following equation:
    \begin{equation}
    \begin{array}{rl}
         R(x)=&1-B(x)+\int_{v=a}^{b}\int_{t=0}^{x}\int_{y=0}^{\infty}R((1+v)(x-t)+y)\mu e^{-\mu y}dB(t)\frac{dv}{b-a}.
    \end{array}\label{eq11a}
\end{equation}
Applying Laplace transform in \eqref{eq11a} yields
\begin{displaymath}
    \rho(s)=\frac{1-\phi(s)}{s}+I^{*}(s),
\end{displaymath}
where
\begin{displaymath}
    \begin{array}{rl}
         I^{*}(s)=&\frac{\phi(s)}{b-a}\int_{v=a}^{b}\int_{x=t}^{\infty}e^{-(x-t)(s-\mu(1+v))}\int_{z=(1+v)(x-t)}^{\infty}\mu e^{-\mu z}R(z)dzdxdv\vspace{2mm}  \\
         =&\frac{\phi(s)}{b-a} \int_{v=a}^{b}\int_{z=0}^{\infty}\mu e^{-\mu z}R(z)\frac{e^{-\frac{z}{1+v}(s-\mu(1+v))}-1}{\mu(1+v)-s}dzdv\vspace{2mm}  \\
         =&\frac{\phi(s)}{b-a} \int_{v=a}^{b}\frac{\mu}{\mu(1+v)-s}(\rho(\frac{s}{1+v})-\rho(\mu))dv\vspace{2mm}  \\
         =&\frac{\phi(s)}{b-a} \int_{v=a}^{b}\frac{\mu}{\mu(1+v)-s}\rho(\frac{s}{1+v})dv-\frac{\phi(s)}{b-a}\rho(\mu)\ln(\frac{(1+b)\mu-s}{(1+a)\mu-s}).
    \end{array}
\end{displaymath}
Thus, we come up with the following functional equation:
\begin{equation}
    \rho(s)=\int_{a}^{b}k(s,v_{1})\rho(\frac{s}{1+v_{1}})dv_{1}+L(s),\label{funn1}
\end{equation}
where,
\begin{displaymath}
    \begin{array}{rl}
         k(s,v):=& \frac{\phi(s)}{b-a}\frac{\mu}{\mu(1+v)-s}, \\
         L(s):=& \frac{1-\phi(s)}{s}-\frac{\phi(s)}{b-a}\rho(\mu)\ln(\frac{(1+b)\mu-s}{(1+a)\mu-s}).
    \end{array}
\end{displaymath}
Iterating $n$ times yields
\begin{equation}
    \begin{array}{rl}
         \rho(s)=&\int\ldots\int_{[a,b]^{n+1}}\prod_{j=1}^{n+1}k(\frac{s}{\prod_{i=1}^{j-1}(1+v_{i})},v_{j})\rho(\frac{s}{\prod_{m=1}^{n+1}(1+v_{m})})dv_1\ldots dv_{n+1}  \vspace{2mm}\\
         &+L(s)+\sum_{j=1}^{n} \int\ldots\int_{[a,b]^{j}}\prod_{i=1}^{j}k(\frac{s}{\prod_{m=1}^{i-1}(1+v_{m})},v_{i})L(\frac{s}{\prod_{m=1}^{j}(1+v_{m})})dv_1\ldots dv_{j},
    \end{array}\label{iui}
\end{equation}
with the convention that an empty product equals one. In the following, we will let $n$ tend to $\infty$ to obtain an expression for $\rho(s)$. Thus, we need to estimate the limit of the first term on the right-hand side of \eqref{iui}, as well as to verify the convergence of the summation in the third term. To cope with this task is the evaluation of
\begin{displaymath}
    |\int\ldots\int_{[a,b]^{n+1}}\prod_{j=1}^{n+1}k(\frac{s}{\prod_{i=1}^{j-1}(1+v_{i})},v_{j})dv_1\ldots dv_{n+1}  |.
\end{displaymath}
Note that $k(s,v)=\frac{1}{(1+v)(b-a)}\phi(s)A(-\frac{s}{1+v})=\frac{1}{(1+v)(b-a)}E(e^{-s B+\frac{s}{1+v} A})$, where $A(s):=E(e^{-sA})$, $A\sim\exp(\mu)$, $E(e^{-sB})=\phi(s)$. Thus,
\begin{displaymath}
    \begin{array}{l}
         |\int\ldots\int_{[a,b]^{n+1}}\prod_{j=1}^{n+1}\frac{1}{(1+v_{j})(b-a)}\phi(\frac{s}{\prod_{i=1}^{j-1}(1+v_{i})})A(-\frac{s}{1+v_{j}})dv_1\ldots dv_{n+1}|\vspace{2mm}\\  \leq |\int\ldots\int_{[a,b]^{n+1}}\prod_{j=1}^{n+1}\frac{1}{1+v_{j}}dv_1\ldots dv_{n+1}  |=|\int_{[a,b]}\frac{1}{1+v_{1}}dv_1\ldots\int_{[a,b]}\frac{1}{1+v_{n+1}}dv_{n+1}|=(\ln(\frac{b+1}{a+1}))^{n+1}.
         \end{array}
\end{displaymath}
Having in mind that $|\rho(s)|\leq 1$, the magnitude of the first term on the right-hand of \eqref{iui} is no more than $(\ln(\frac{b+1}{a+1}))^{n+1}$, which tends to 0 as $n\to\infty$, provided that $\ln(\frac{b+1}{a+1})<1$, or equivalently $\frac{b+1}{a+1}<e$.  Therefore,
\begin{equation}
    \begin{array}{rl}
         \rho(s)=&L(s)+\sum_{j=1}^{\infty} \int\ldots\int_{[a,b]^{j}}\prod_{i=1}^{j}k(\frac{s}{\prod_{m=1}^{i-1}(1+v_{m})},v_{i})L(\frac{s}{\prod_{m=1}^{j}(1+v_{m})})dv_1\ldots dv_{j}.
    \end{array}\label{iuis}
\end{equation}
Finally, $\rho(\mu)$ can be derived by substituting $s=\mu$ in \eqref{iuis}, and solving the derived equation. More precisely,
\begin{displaymath}
\begin{array}{rl}
    \rho(\mu)=&\sum_{j=0}^{\infty} \int\ldots\int_{[a,b]^{j}}\prod_{i=1}^{j}k(\frac{\mu}{\prod_{m=1}^{i-1}(1+v_{m})},v_{i})\vspace{2mm}\\&\times\left[\frac{1-\phi(\frac{\mu}{\prod_{m=1}^{j}(1+v_{m})})}{\frac{\mu}{\prod_{m=1}^{j}(1+v_{m})}}-\phi(\frac{\mu}{\prod_{m=1}^{j}(1+v_{m})})\frac{\rho(\mu)}{b-a}\ln\left(\frac{(1+b)\mu-\frac{\mu}{\prod_{m=1}^{j}(1+v_{m})}}{(1+a)\mu-\frac{\mu}{\prod_{m=1}^{j}(1+v_{m})}}\right)\right] dv_1\ldots dv_{j}.\end{array}
\end{displaymath}
Hence,
\begin{displaymath}
    \rho(\mu)=\frac{\sum_{j=0}^{\infty} \int\ldots\int_{[a,b]^{j}}\prod_{i=1}^{j}k(\frac{\mu}{\prod_{m=1}^{i-1}(1+v_{m})},v_{i})\frac{1-\phi(\frac{\mu}{\prod_{m=1}^{j}(1+v_{m})})}{\frac{\mu}{\prod_{m=1}^{j}(1+v_{m})}}dv_1\ldots dv_{j}}{1+\frac{1}{b-a}\sum_{j=0}^{\infty} \int\ldots\int_{[a,b]^{j}}\prod_{i=1}^{j}k(\frac{\mu}{\prod_{m=1}^{i-1}(1+v_{m})},v_{i})\phi(\frac{\mu}{\prod_{m=1}^{j}(1+v_{m})})\ln\left(\frac{(1+b)\prod_{m=1}^{j}(1+v_{m})-1}{(1+a){\prod_{m=1}^{j}(1+v_{m})}-1}\right) dv_1\ldots dv_{j}}.
\end{displaymath}
\begin{remark}
    Clearly, the approach we followed in this section can be used to incorporate dependencies among the gain interarrivals and the gain sizes, as well as the case where gain sizes follow Erlang, or mixed Erlang distribution.
\end{remark}
\section{Conclusion \& future work}
In this work, we cope with several non-trivial generalizations of the dual risk model with proportional gains, for which several independent assumptions among the gain interarrivals and the gain sizes were lifted, but still, we are able to derive explicit expressions for the ruin probability and the time to ruin. Among others we considered causal dependencies as well as dependencies that are based on the FGM copula. On top of that we also consider cases where the gain size is no longer exponentially distributed, but Erlang or mixed Erlang, which result in additional interesting observations. Moreover, it is well known that the mixed Erlang distribution belongs to a class of the phase-type distributions, which is dense in the space of distribution functions
defined on $[0,\infty)$. Moreover, we cope with the case where the proportional parameter is a uniformly distributed random variable. In the future we plan to investigate several open tasks:
\begin{enumerate}
    \item To identify the value of the discounted cumulative dividend payments in the presence of dependence among gain interarrivals and gain sizes.
    \item To cope with the problem where the jumps up from the level $u$ have a more general form.
    \item To consider more sophisticated copulas that describe the dependence structure, but still, to be able to derive explicit results for the values of interest. Finally, it would be interesting to consider a semi-Markovian dual risk model.
    \item Other opportunities for future study refer to multidimensional surplus processes, although are anticipated to be highly challenging.
\end{enumerate}
\section*{Competing interest} The author(s) declare none.
\bibliographystyle{abbrv}
\bibliography{dualrisk}
\end{document}